\documentclass[10pt,oneside,a4paper]{amsart}
\usepackage{amsmath,amsfonts,amsthm, amssymb}
\usepackage{array}
\usepackage[all]{xy}
\usepackage{portland}

\theoremstyle{plain}
\newtheorem{theorem}{Theorem}[section]
\newtheorem{proposition}[theorem]{Proposition}
\newtheorem{lemma}[theorem]{Lemma}
\newtheorem{corollary}[theorem]{Corollary}

\theoremstyle{definition}
\newtheorem{definition}[theorem]{Definition}

\theoremstyle{remark}

\newtheorem{remark}[theorem]{Remark}
\newtheorem{example}[theorem]{Example}
\newtheorem{examples}[theorem]{Examples}

\newcommand{\ovl}{\overline}
\newcommand{\Kern}{\mathrm{Ker}}

\renewcommand{\lim}{\mathrm{lim}}
\newcommand{\colim}{\mathrm{colim}}

\newcommand{\sh}{\mathrm{sh}}
\newcommand{\ex}{\mathrm{ex}}
\newcommand{\dg}{\mathrm{dg}}
\newcommand{\hoch}{\mathrm{hoch}}

\newcommand{\gro}{\mathrm{gro}}

\newcommand{\mac}{\mathrm{mac}}
\newcommand{\inj}{\mathsf{inj}}
\newcommand{\add}{\mathrm{add}}

\newcommand{\free}{\mathsf{free}}

\newcommand{\Ext}{\mathrm{Ext}}

\newcommand{\Hom}{\mathrm{Hom}}

\newcommand{\RHom}{\mathrm{RHom}}

\newcommand{\op}{^{\mathrm{op}}}
\newcommand{\Ob}{\mathrm{Ob}}
\newcommand{\Z}{\mathbb{Z}}

\newcommand{\AAA}{\mathfrak{a}}

\newcommand{\UUU}{\mathfrak{u}}

\newcommand{\CC}{\mathbf{C}}
\newcommand{\Set}{\ensuremath{\mathsf{Set}} }

\newcommand{\Ab}{\ensuremath{\mathsf{Ab}} }
\newcommand{\Mod}{\ensuremath{\mathsf{Mod}} }

\newcommand{\Lex}{\ensuremath{\mathsf{Lex}} }

\newcommand{\Sh}{\ensuremath{\mathsf{Sh}} }

\newcommand{\Ind}{\ensuremath{\mathsf{Ind}}}
\newcommand{\Fun}{\ensuremath{\mathsf{Fun}}}
\newcommand{\Add}{\ensuremath{\mathsf{Add}}}

\newcommand{\lra}{\longrightarrow}

\newcommand{\aaa}{\ensuremath{\mathcal{A}}}
\newcommand{\bbb}{\ensuremath{\mathcal{B}}}
\newcommand{\ccc}{\ensuremath{\mathcal{C}}}
\newcommand{\ddd}{\ensuremath{\mathcal{D}}}

\newcommand{\LLL}{\ensuremath{\mathcal{L}}}

\newcommand{\sss}{\ensuremath{\mathcal{S}}}
\newcommand{\ttt}{\ensuremath{\mathcal{T}}}

\newcommand{\www}{\ensuremath{\mathcal{W}}}

\CompileMatrices
\SelectTips{cm}{11}

\newcommand{\idot}{{\:\raisebox{1pt}{\text{\circle*{1.5}}}}}
\newcommand{\hdot}{{\:\raisebox{3pt}{\text{\circle*{1.5}}}}}
\newcommand{\amod}{{\text{\rm -mod}}}
\newcommand{\Spec}{\operatorname{Spec}}
\newcommand{\id}{{\sf id}}

\title{Cohomology of exact categories and (non-)additive sheaves}
\author{Dmitry Kaledin} 
\address[Dmitry Kaledin]{Steklov Math Institute, Moscow, USSR}
\email{kaledin@mi.ras.ru}
\author{Wendy Lowen} 
\address[Wendy Lowen]{Departement Wiskunde-Informatica, Middelheimcampus,
Middelheimlaan 1,
2020 Antwerp, Belgium}
\email{wendy.lowen@ua.ac.be}
\thanks{The first author has been partially supported by AG Laboratory
    SU-HSE, RF government grant, ag. 11.G34.31.0023, the RFBR grant 09-01-00242 and the Science
Foundation of the
SU-HSE award No. 10-09-0015. The second author is postdoctoral fellow with the Research Foundation - Flanders (FWO) and acknowledges the support of the European Union, ERC grant No 
257004-HHNcdMir}

\begin{document}
\maketitle

\begin{abstract}
We use (non-)additive sheaves to introduce an (absolute) notion of Hochschild cohomology for exact categories as $\Ext$'s in a suitable bisheaf category. We compare our approach to various definitions present in the literature.
\end{abstract}

\section{Introduction}

Given an associative algebra $A$ over a field $k$, one can define
its Hochschild homology groups $HH_\idot(A)$ and its Hochschild
cohomology groups $HH^\hdot(A)$. Non-commutative geometry, in its
homological version, starts with the observation that Hochschild
homology classes behave ``as differential forms'', while Hochschild
cohomology classes are similar to vector fields. When $A$ is
commutative and $\Spec A$ is a smooth algebraic variety over $k$,
this observation becomes a precise theorem, namely, the famous
theorem of Hochschild, Kostant and Rosenberg \cite{HKR}. In the general case,
both $HH_\idot(A)$ and $HH^\hdot(A)$ still carry some additional
structures analogous to what one finds for a commutative algebra. For
$HH_\idot(A)$, the relevant structure is the Connes-Tsygan
differential $B$ which gives rise to cyclic homology -- this is
analogous to the de Rham differential. For $HH^\hdot(A)$, the
structure is the so-called Gerstenhaber bracket which turns
$HH^\hdot(A)$ into a Lie algebra -- this is analogous to the Lie
bracket of vector fields. There are certain natural compatibilities
between the bracket and the differential, axiomatized by Tsygan and
Tamarkin under the name of ``non-commutative calculus'' \cite{tamarkintsygan}.

\medskip

If one thinks of an algebra $A$ as a simple example of a
``non-commutative algebraic variety'', then Hochschild homology
usually gives rise to homological invariants of the variety, such as
e.g. de Rham or cristalline cohomology. Hochschild cohomology, on
the other hand, is intimately related to automorphisms and
deformations of $A$.

\medskip

For real-life applications, it is highly desirable to extend the
basic theory of Hochschild homology and cohomology to ``more
general'' non-commutative varieties. This can mean different things
in different contexts; but at the very least, one should be able to
develop the theory for an abelian category $\ccc$ (a motivating
observation here is that if two algebras $A$, $B$ have equivalent
categories $A\amod \cong B\amod$ of left modules, then their
Hochschild homology and cohomology are canonically identified). For
Hochschild homology, this has been accomplished in a more-or-less
exhaustive fashion by B. Keller \cite{keller} back in the
1990ies. For Hochschild cohomology, the story should be simpler:
morally speaking, the Hochschild cohomology algebra $HH^\hdot(\ccc)$
should just be the algebra of $\Ext$'s from the identity endofunctor
of $\ccc$ to itself. However, finding an appropriate category where
these $\Ext$'s can be computed is a delicate matter. 

Perhaps because of this, the rigorous cohomological theory appeared
later than the homological one; essentially, it was started in
\cite{lowenvandenbergh1}, \cite{lowenvandenbergh2}, where a
Hochschild cohomology theory for abelian categories is constructed,
and its relation to deformations of the category is discussed.  But
unfortunately, the theory that exists so far is closely modeled on
the theory for associative algebras. As a result, it lacks some
essential features which should in fact become automatic in the
consistently categorical approach. This becomes quite obvious when
one tries to apply the theory to concrete problems; for one example
of this, we refer the reader to \cite{kaledin1}, where the
application intended is to Gabber's involutivity theorem.

The present paper arose as an attempt to at least fill the gaps
noted in \cite{kaledin1}, and at most, to sketch a more-or-less
comprehensive theory of Hochschild cohomology of abelian categories
and its relation to deformations. As it happens, already the
definitions of Hochschild cohomology, when done accurately, take up
quite a lot of space. This is as far as we get in this paper,
relegating both the Gerstenhaber bracket and the deformation theory
story to subsequent work.

One additional thing that emerges clearly in the categorical
approach is the ability to work ``absolutely'', not over a fixed
field $k$. The motivating example here is very basic: the category
of vector spaces over $\Z/p\Z$ has a natural ``first-order
deformation'' to the category of modules over $\Z/p^2\Z$. A truly
comprehensive Hochschild cohomology theory for abelian categories
should include this example, and assign to it a non-trivial
deformation class. Some of the theories we construct in the present
paper should be able to do this. In order to achieve this, we have to
spend quite a lot of time on foundations, but we believe that
ultimately, this is time well spent.

\medskip

The general outline of the paper is as follows. As we have already
noted, the definition of Hochschild cohomology should be obvious
once one has an appropriate category of endofunctors of our abelian
category $\ccc$. If $\ccc$ is the category of modules over an
algebra $A$, then a natural candidate for its endofunctor category
is the category of bimodules over the same algebra -- this is what
gives the classical Hochschild cohomology. An ``absolute'' version
of this story also exists, and it has been known for quite some time
now, starting from \cite{jibladzepirashvili1}. However, the
situation for a general abelian category $\ccc$ turns out to be
somewhat delicate. In Section~\ref{sec.1}, we discuss in some detail
various embedding theorems which allow one to represent an abelian
category $\ccc$ as a category of sheaves on itself, then define its
endofunctor category as a category of sheaves on $\ccc^{\op} \times
\ccc$, and so on. In the module category case, everything is known,
but we reproduce the results for the convenience of the reader; the
general case requires using appropriate Grothendieck topologies,
and this seems to be new. As an unexpected bonus, we discover along the way
that a very natural relaxation of some conditions produces exactly
the exact categories in the sense of Quillen, so that the whole
story generalizes to exact categories without any changes at all. In
Section~\ref{sec.2}, we discuss the derived versions of the sheaf
categories and various exactness properties of natural functors
between them. Then in Section~\ref{sec.3}, we are finally able to
introduce Hochschild cohomology. We also discuss other definitions
present in the literature and prove various comparison theorems
between them. Of course, to be useful, such a list of comparison
theorems should be exhaustive; this we have strived to achieve, to
the best of our knowledge. In particular, we do treat the absolute
case -- the relevant notion here is Mac Lane homology and
cohomology. Finally, in the last section, we discuss informally what
does {\em not} work in our approach, especially in the absolute
case, and what is the relation between our work and more abstract
theory based on various triangulated category enhancements.

\bigskip

\noindent \emph{Acknowledgement.}
Over the years, we have had an opportunity to discuss Hochschild
homology and cohomology with many people, and benefited a lot from
these discussions. The first author is particularly grateful for the
things he has learned from M. Kontsevich, L. Hesselholt,
D. Tamarkin, and B. Tsygan.
The second author is grateful to Ragnar Buchweitz for bringing the
definition of Hochschild cohomology in \S \ref{hhtop} to her
attention and for explaining the proof of Theorem \ref{maingroth}
for module categories. She is also grateful to Michel Van den Bergh
for numerous things she learned from him, some dating back to her
master thesis - under his supervision - on Grothendieck categories
and additive sheaves.
After the first version of the paper was posted to the web, we have
received several valuable comments and explanations from T. Pirashvili,
for which we are very grateful.

\section{Sheaf categories}\label{sec.1}

This section contains some basic notions and facts concerning sheaves taking values in the category $\Ab$ of abelian groups. The setting in which we will work is that of single morphism topologies, i.e. topologies for which covers are determined by the morphisms in a certain collection $\Lambda$. Our main application is to exact categories $\ccc$, for which $\ccc$ comes naturally equipped with the single deflation topology, and $\ccc^{\op}$ with the single inflation topology. In this context, we introduce a number of bifunctor categories consisting of bifunctors that are additive in some of the variables and sheaves in some of the variables.

\subsection{Additive topologies}\label{paraddtop}
In this section we mainly fix some notations and terminology. For categories $\ccc$, $\ddd$ with $\ccc$ small we denote by $\Fun(\ccc, \ddd)$ the category of functors from $\ccc$ to $\ddd$, and we put $\Fun(\ccc) = \Fun(\ccc^{\op}, \Ab)$. For $\Z$-linear categories $\ccc$, $\ddd$ with $\ccc$ small we denote by $\Add(\ccc, \ddd)$ the category of additive functors from $\ccc$ to $\ddd$, and we put $\Mod(\ccc) = \Add(\ccc^{\op}, \Ab)$. Objects of $\Fun(\ccc)$ are called functors while objects of $\Mod(\ccc)$ are called modules. By a \emph{topology} on a small category $\ccc$ we mean a Grothendieck topology. On a small \emph{$\Z$-linear} category we will also use the parallel enriched notion of an \emph{additive} topology (see \cite{borceuxquinteiro}, \cite{popescu}, \cite{lowen1}). This is obtained from the usual notion of a Grothendieck topology by replacing $\Set$ by $\Ab$ and $\Fun(\ccc^{\op}, \Set)$ by $\Mod(\ccc)$.
More precisely:
\begin{definition}
An \emph{additive topology} $\ttt$ on a small $\Z$-linear category $\ccc$ is given by specifying for each object $C \in \ccc$ a collection $\ttt(C)$ of submodules of $\ccc(-,C) \in \Mod(\ccc)$ satisfying the following axioms:
\begin{enumerate}
\item $\ccc(-,C) \in \ttt(C)$.
\item For $R \in \ttt(C)$ and $f: D \lra C$ in $\ccc$ the pullback $f^{-1}R$ in $\Mod(\ccc)$ of $R$ along $f\circ -: \ccc(-,D) \lra \ccc(-,C)$ is in $\ttt(D)$.
\item Consider $S \in \ttt(C)$ and an arbitrary submodule $R \subseteq \ccc(-,C)$. If for every $D \in \ccc$ and $f \in S(D)$ the pullback $f^{-1}R$ is in $\ttt(D)$, then it follows that $R \in \ttt(C)$.
\end{enumerate}
\end{definition}

An additive topology on a one-object $\Z$-linear category corresponds precisely to a Gabriel topology on a ring \cite{gabriel}.

As usual, a submodule $R \subseteq \ccc(-,C)$ is identified with the set $\coprod_{D \in \ccc}R(D) \subseteq \coprod_{D \in \ccc}\ccc(D,C)$, i.e. $R$ is considered as an ``additive sieve''. A submodule $R \in \ttt(C)$ is called a \emph{cover (of $C$)}. An additive topology $\ttt$ on $\ccc$ determines a Grothendieck category $\Sh^{\mathrm{add}}(\ccc, \ttt) \subseteq \Mod(\ccc)$ of \emph{additive sheaves}, i.e. modules $F \in \Mod(\ccc)$ such that every cover $R \subseteq \ccc(-,C)$ in $\ttt(C)$ induces a bijection
$$F(C) \cong \Mod(\ccc)(\ccc(-C), F) \lra \Mod(\ccc)(R, F).$$
Conversely any Grothendieck category $\aaa$ can be represented as an additive sheaf category for suitable choices of $\ccc$ (see \cite{lowen1}), the easiest choice for $\ccc$ being a full generating subcategory as in the Gabriel-Popescu theorem \cite{popescugabriel}.

\subsection{Single morphism topologies}
Let $\ccc$ be a small (resp. small $\Z$-linear) category and $\Lambda$ a collection of $\ccc$-morphisms. We define a subfunctor (resp. a submodule) $R \subseteq \ccc(-,C)$ to be a \emph{$\Lambda$-cover} if $R$ (considered as a sieve) contains a morphism $\lambda \in \Lambda$. If the $\Lambda$-covers define a topology $\ttt_{\Lambda}$ (resp. an additive topology $\ttt_{\Lambda}^{\mathrm{add}}$) on $\ccc$, then this topology is called the \emph{single $\Lambda$-topology} (resp. the \emph{additive single $\Lambda$-topology}). 

Let us now spell out what it means for $F \in \Fun(\ccc^{\op}, \Set)$ to be a sheaf for $\ttt_{\Lambda}$. For $\lambda: D \lra C$ in $\Lambda$, a compatible family of elements with respect to the cover $\langle \lambda \rangle$ generated by $\lambda$ corresponds to an element $x \in F(D)$ such that for every commutative diagram
$$\xymatrix{ D \ar[r]^{\lambda} & C \\ E \ar[u]^{\alpha_1} \ar[r]_{\alpha_2} & D \ar[u]_{\lambda}}$$
we have $F(\alpha_1)(x) = F(\alpha_2)(x)$. Hence, the sheaf property with respect to $\lambda$ says that for such an element $x \in F(D)$ there is a unique element $y \in F(C)$ with $F(\lambda)(y) = x$.

Recall that a filtered colimit $\colim_i F_i$ is called \emph{monofiltered} if all the transition morphisms $F_i \lra F_j$ are monomorphisms.

\begin{lemma}\label{moncolim}
A monofiltered colimit of sheaves (in $\Fun(\ccc^{\op}, \Set)$)  remains a sheaf.
\end{lemma}

\begin{proof}
Consider a monofiltered colimit $\colim_i F_i$ of sheaves $F_i$ and $\lambda: D \lra C$ in $\Lambda$. Suppose $x \in \colim_i F_i(D)$ is compatible and let $x_i \in F_i(D)$ be a representative of $x$. Consider $\alpha_1, \alpha_2: E \lra D$ with $\lambda \alpha_1 = \lambda \alpha_2$. Now $F_i(\alpha_1)(x_i)$ and $F_i(\alpha_2)(x_i)$ become equal in $\colim_i F_i(E)$, but since this colimit is monofiltered, we obtain $F_i(\alpha_1)(x_i) = F_i(\alpha_2)(x_i)$in $F_i(E)$. Hence, $x_i$ is compatible and there exists $y_i \in F_i(C)$ with $F(\lambda)(y_i) = x_i$. Furthermore, if $y, z \in \colim_i F_i(C)$ become equal in $\colim_i F_i(D)$, appropriate representatives $y_i, z_i \in F_i(C)$ become equal in $F_i(D)$, and hence $y_i = z_i$ and $y = z$.
\end{proof}

If $\ccc$ is small $\Z$-linear, it makes sense to consider both $\ttt_{\Lambda}$ and $\ttt_{\Lambda}^{\mathrm{add}}$ on $\ccc$. The subfunctors $R = \langle \lambda \rangle \subseteq \ccc(-,C)$ of morphisms factoring through a given $\lambda \in \Lambda$ are additive (whence submodules) and constitute a basis for both $\ttt_{\Lambda}$ and $\ttt_{\Lambda}^{\mathrm{add}}$.

We are mainly interested in sheaves taking values in the category $\Ab$ of abelian groups.
Consider $\Fun(\ccc) = \Fun(\ccc^{\op}, \Ab)$, the category  $\Sh_{\Lambda}(\ccc) = \Sh(\ccc, \ttt_{\Lambda}) \subseteq \Fun(\ccc)$ of (non-additive) sheaves of abelian groups on $\ccc$, $\Mod(\ccc) = \Add(\ccc^{\op}, \Ab)$ and the category $\Sh_{\Lambda}^{\mathrm{add}}(\ccc) =\Sh^{\mathrm{add}}(\ccc, \ttt_{\Lambda}^{\mathrm{add}}) \subseteq \Mod(\ccc)$ of additive sheaves on $\ccc$. By the previous observations, we have
$$\Sh_{\Lambda}^{\mathrm{add}}(\ccc) = \Sh_{\Lambda}(\ccc) \cap \Mod(\ccc).$$

Recall that an object $A$ in a category $\aaa$ is \emph{finitely generated} if $\aaa(A,-): \aaa \lra \Set$ commutes with monofiltered colimits.
We have the following natural source of single morphism topologies:

\begin{proposition}\label{evgen}
Let $\aaa$ be a Grothendieck category and $\ccc \subseteq \aaa$ a small full additive subcategory.
The following are equivalent:
\begin{enumerate}
\item The objects of $\ccc$ are finitely generated generators of $\aaa$.
\item $\Lambda = \{ \lambda \in \ccc \,\, |\,\, \lambda \,\, \text{is an epimorphism in}\,\, \aaa \}$ defines an additive single $\Lambda$-topology $\ttt_{\Lambda}^{\mathrm{add}}$ on $\ccc$ with
$$\Sh(\ccc, \ttt_{\Lambda}^{\mathrm{add}}) \cong \aaa.$$
\end{enumerate}
\end{proposition}

\begin{proof}
If (1) holds, there is an additive topology $\ttt$ on $\ccc$ with $\Sh(\ccc, \ttt) \cong \ccc$ and for this topology $R \subseteq \ccc(-,C)$ is a cover if and only if $\oplus_{f \in R}C_f \lra C$ is an epimorphisms in $\aaa$. Since $C$ is finitely generated, there are finitely many morphisms $f_i: C_i \lra C$ in $R$, $i = 1, \dots n$, for which $f = \sum_{i = 1}^n f_i: \oplus_{i = 1}^n C_i \lra C$ is an epimorphism.  But since $R$ is an additive subfunctor, in fact $f \in R$. Conversely, suppose (2) holds. Obviously $\ccc$ generates $\aaa$, so we are to show that $C \in \ccc$ is finitely generated in $\Sh(\ccc, \ttt_{\Lambda}^{\mathrm{add}})$. This easily follows from the fact that $C$ is finitely generated in $\Mod(\ccc)$ and Lemma \ref{moncolim}.
\end{proof}

\begin{remark}\label{remsplit}
If all the morphisms $\lambda \in \Lambda$ of Proposition \ref{evgen} become split epimorphisms in $\aaa$, the topology $\ttt_{\Lambda}^{\mathrm{add}}$ is reduced to the trivial topology with $\Sh(\ccc, \ttt_{\Lambda}^{\mathrm{add}}) = \Mod(\ccc)$. This situation is equivalent to the objects in $\ccc \subseteq \aaa$ being finitely generated \emph{projective} generators in $\aaa$. \end{remark}

Often a collection $\Lambda$ can be directly seen to define single $\Lambda$-topologies:

\begin{proposition}\label{propcond}
Let $\ccc$ be a small category (resp. small $\Z$-linear category) and $\Lambda$ a collection of morphisms such that:
\begin{enumerate}
\item $\Lambda$ contains isomorphisms;
\item For $\lambda: D \lra C$ in $\Lambda$ and $f: C' \lra C$ arbitrary, the pullback $\lambda': D' \lra C'$ exists and is in $\Lambda$;
\item $\Lambda$ is stable under composition.
\end{enumerate}
Then $\Lambda$ defines a single $\Lambda$-topology (resp. an additive single $\Lambda$-topology) on $\ccc$.
\end{proposition}

Suppose $\Lambda$ determines single morphism topologies $\ttt_{\Lambda}$ and $\ttt_{\Lambda}^{\mathrm{add}}$. The inclusions $i': \Sh_{\Lambda}(\ccc) \subseteq \Fun(\ccc)$ and $i:  \Sh_{\Lambda}^{\mathrm{add}}(\ccc) \subseteq \Mod(\ccc)$ have exact left adjoint sheafification functors $a': \Fun(\ccc) \lra \Sh_{\Lambda}(\ccc)$ and $a: \Mod(\ccc) \lra \Sh_{\Lambda}(\ccc)$ respectively. 

\begin{definition}
A functor $F \in \Fun(\ccc)$ is \emph{weakly $\Lambda$-effaceable} if and only if for every $C \in \ccc$ and every $x \in F(C)$, there exists a morphism $\lambda: C' \lra C$ in $\Lambda$ with $F(\lambda)(x) = 0$.
\end{definition}
Let $\www_{\Lambda} \subseteq \Fun(\ccc)$ be the full subcatgory of weakly $\Lambda$-effaceable functors, and $\www^{\mathrm{add}}_{\Lambda} \subseteq \Mod(\ccc)$ the full subcategory of weakly $\Lambda$-effaceable modules. Clearly $\www^{\mathrm{add}}_{\Lambda} = \www_{\Lambda} \cap \Mod(\ccc)$. From the concrete formulae for sheafification and the fact that the $\langle \lambda \rangle$ constitute a basis for $\ttt_{\Lambda}$ and $\ttt_{\Lambda}^{\mathrm{add}}$, it follows that:
$$\www_{\Lambda} = \Kern(a) \hspace{2cm} \www^{\mathrm{add}}_{\Lambda} = \Kern(a').$$
In particular, $\www_{\Lambda}$ and $\www^{\mathrm{add}}_{\Lambda}$ are localizing Serre subcategories of $\Fun(\ccc)$ and $\Mod(\ccc)$ respectively, and
$$\Sh_{\Lambda}(\ccc) = \www_{\Lambda}^{\perp} \hspace{2cm} \Sh^{\mathrm{add}}_{\Lambda}(\ccc) = (\www_{\Lambda}^{\mathrm{add}})^{\perp}$$
where $F \in \www^{\perp} \iff [\forall\,\, W \in \www: \,\,\, \Hom(W,F) = 0 = \Ext^1(W,F)]$ (see for example \cite{krause3}).

We obtain commutative diagrams:
\begin{equation}\label{eqdiag}
\xymatrix{{\Mod(\ccc)} \ar[r]^j \ar@/_10pt/[d]_{a} & {\Fun(\ccc)} \ar@/_10pt/[d]_{a'}\\
{\Sh^{\mathrm{add}}_{\Lambda}(\ccc)} \ar[u]_{i} \ar[r]_{j'} & {\Sh_{\Lambda}(\ccc).} \ar[u]_{i'}}
\end{equation}

\begin{lemma}
In the above diagram, $j'$ is an exact functor.
\end{lemma}

\begin{proof}
Consider an exact sequence $0 \lra A \lra B \lra C \lra 0$ in $\Sh^{\mathrm{add}}_{\Lambda}(\ccc)$. In particular, $0 \lra i(A) \lra i(B) \lra i(C)$ is exact in $\Mod(\ccc)$ and we can complete it into an exact sequence $0 \lra i(A) \lra i(B) \lra i(C) \lra M \lra 0$ in $\Mod(\ccc)$. Since $a$ is exact, this implies $a(M) = 0$, or in other words $M \in \www_{\Lambda}^{\mathrm{add}}$. But $j$, being obviously exact, maps this sequence to the exact sequence $0 \lra ji(A) \lra ji(B) \lra ji(C) \lra j(M) \lra 0$ in $\Fun(\ccc)$. Since $a'(j(M)) = j'(a(M)) = 0$ we have an exact sequence $0 \lra j'(A) \lra j'(B) \lra j'(C) \lra 0$ in $\Sh_{\Lambda}(\ccc)$ as desired.
\end{proof}

\begin{remark}
Note that the inclusion $j: \Mod(\ccc) \lra \Fun(\ccc)$ has a left adjoint ``additivization'' functor which is not exact. Consequently, it is impossible to express additivity of functors by means of a topology on $\ccc$.
\end{remark}

\subsection{Additive sheaves inside non-additive sheaves}
Let $\ccc$ be a small additive category. It is well known that the inclusion 
$$j: \Mod(\ccc) \subseteq \Fun(\ccc)$$
is an exact embedding and a Serre subcategory (see e.g. \cite{pirashvili} and the references therein). In this section we extend the result to the inclusion 
$$j': \Sh^{\mathrm{add}}(\ccc) \lra \Sh(\ccc)$$
in case $\Lambda$ determines single morphism topologies $\ttt$ and $\ttt^{\mathrm{add}}$ on $\ccc$ (we suppress $\Lambda$ in all notations). 
The ingredients of the proof are well known, but we include them for completeness.

 
We start with the following observation:

\begin{lemma}\label{lemaddext}
The inclusion $j: \Mod(\ccc) \lra \Fun(\ccc)$ is an exact embedding which is closed under extensions. 
\end{lemma}

\begin{proof}
Since $\ccc$ is an additive category, $j$ is fully faithful. That $\Mod(\ccc)$ is closed in $\Fun(\ccc)$ under extensions easily follows from the 5-lemma. 
\end{proof}

Next we extend Lemma \ref{lemaddext} to sheaves:

\begin{proposition}\label{ext1}
The inclusion $j': \Sh^{\add}(\ccc) \lra \Sh(\ccc)$ is an exact embedding which is closed under extensions.
\end{proposition}

\begin{proof}
Consider an exact sequence $0 \lra F' \lra F \lra F'' \lra 0$ in $\Sh(\ccc)$ with $F', F'' \in \Sh^{\add}(\ccc)$. This means that we have an exact sequence 
$$0 \lra F' \lra F \lra F'' \lra W \lra 0$$
in $\Fun(\ccc)$ in which $F, F''$ are additive and $W$ is weakly effaceable. We are to show that $F$ is additive. 
By Lemma \ref{Wnul}, $W(0) = 0$ and hence also $F(0) = 0$. It remains to show that for $A, B \in \ccc$, the canonical map $$\eta: F(A \oplus B) \lra F(A) \oplus F(B)$$ is an isomorphism.
By Lemma \ref{nuloplus}, $\eta$ is an epimorphism. Furthermore, from the diagram
$$\xymatrix{ 0 \ar[r] & {F'(A \oplus B)} \ar[r] \ar[d]_{\cong} & {F(A \oplus B)} \ar[r] \ar[d]_{\eta} & {F''(A \oplus B)} \ar[d]^{\cong} \\ 0 \ar[r] & {F'(A) \oplus F'(B)} \ar[r] & {F(A) \oplus F(B)} \ar[r] & {F''(A) \oplus F''(B)}}$$
we deduce that $\eta$ is also a monomorphism.
\end{proof}

\begin{lemma}\label{Wnul}
Suppose $W \in \Fun(\ccc)$ is weakly effaceable. Then $W(0) = 0$. 
\end{lemma}

\begin{proof}
Consider an element $x \in W(0)$. There exists a $\Lambda$-morphism $C \lra 0$ such that $W(0) \lra W(C)$ maps $x$ to $0$. But the map $W(C) \lra W(0)$ induced by $0 \lra C$, being a morphism of abelian groups, maps $0$ to $0$. Since $W(0) \lra W(C) \lra W(0)$ is the identity, this proves that $x = 0$ and consequently $W(0) = 0$. 
\end{proof}

\begin{lemma}\label{nuloplus}
If $F \in \Fun(\ccc)$ satisfies $F(0) = 0$, then for $A, B \in \ccc$ the canonical morphism
$$F(A) \oplus F(B) \lra F(A \oplus B) \lra F(A) \oplus F(B)$$
is equal to the identity.
\end{lemma}

\begin{proof}
Let $s_A$, $s_B$, $p_A$, $p_B$ denote the canonical injections and projections associated to $A \oplus B$. Then we are now dealing with their images under $F$. We have $F(p_A)F(s_A) = F(p_As_A) = F(1_A) = 1_{F(A)}$ and likewise for $B$. Moreover, since $F(0) = 0$, we also have $F(p_A)F(s_B) = F(p_As_B) = F(0) = 0$ and similarly for $F(p_B)F(s_A)$. This finishes the proof.
\end{proof}

\begin{theorem}\label{lexserre}
Let $\ccc$ be a small additive category.
The inclusions
$$j: \Mod(\ccc) \subseteq \Fun(\ccc)$$
and $$j': \Sh^{\add}(\ccc) \subseteq \Sh(\ccc)$$
are Serre subcategories.
\end{theorem}

\begin{proof}
We already showed in Lemma \ref{lemaddext} and Proposition \ref{ext1} that both inclusions are abelian subcategories that are closed under extensions. We need to show that they are closed under subquotients. First, consider an exact sequence
$$0 \lra F' \lra F \lra F'' \lra 0$$
in $\Fun(\ccc)$ in which $F$ is additive. First of all, $F'(0)$ and $F''(0)$ are zero as a subobject and a quotient object of $F(0) = 0$. Now consider morphisms $a, b: C \lra C'$ in $\ccc$ and consider $f' = F'(a + b) - F'(a) - F'(b)$, $f = F(a + b) -F(a) - F(b)$ and $f'' = F''(a + b) - F''(a) - F''(b)$. Then the commutative diagram
$$\xymatrix{0 \ar[r] & {F'(C')} \ar[d]_{f'} \ar[r] & {F(C')} \ar[d]^f \ar[r] & {F''(C')} \ar[d]^{f''} \ar[r] & 0\\
0 \ar[r] & {F'(C)} \ar[r] & {F(C)} \ar[r] & {F''(C)} \ar[r] & 0}$$
immediately yields that $f = 0$ implies that both $f' = 0$ and $f'' = 0$.

For the second claim, consider an exact sequence $0 \lra F' \lra F \lra F'' \lra 0$ in $\Sh(\ccc)$ and suppose that $F$ is additive. Then $F'' = a(Q)$ where $0 \lra F' \lra F \lra Q \lra 0$ is exact in $\Fun(\ccc)$ and $a$ is sheafification. We just obtained that both $F'$ and $Q$ are additive. Hence also $F'' = a(Q)$ is additive.
\end{proof}

\subsection{Single morphism topologies with kernels}
Let $\ccc$ be a small category and suppose $\Lambda$ determines a single $\Lambda$-topology. 
Suppose moreover that the morphisms in $\Lambda$ have kernel pairs. In this situation, the notion of sheaf becomes more tangible. 
For $\lambda \in \Lambda$, consider the kernel pair
$$\xymatrix{ D \ar[r]^{\lambda} & C \\ P \ar[u]^{\kappa_1} \ar[r]_{\kappa_2} & D \ar[u]_{\lambda} }$$
A presheaf $F \in \Fun(\ccc^{\op}, \Set)$  is a sheaf if and only if for every $\lambda \in \Lambda$ with kernel pair $(\kappa_1, \kappa_2)$,
\begin{equation}\label{equalizer}
\xymatrix{{F(C)} \ar[r]_{F(\lambda)} & {F(D)} \ar@/^5pt/[r]^{F(\kappa_1)} \ar@/_5pt/[r]_{F(\kappa_2)} & {F(P)}}
\end{equation}
is an equalizer diagram. 
We immediately deduce the following strengthening of Lemma \ref{moncolim}:

\begin{lemma}\label{filtsheaves}
A filtered colimit of sheaves (in $\Fun(\ccc^{\op}, \Set)$) remains a sheaf. 
\end{lemma}

\begin{example}
If $\ccc$ is a regular category \cite{barr}, then $\Lambda = \{ \lambda \,\, |\,\, \lambda \,\, \text{is a coequalizer}\}$ satisfies the conditions of Proposition \ref{propcond}. 
Since a coequalizer is always the coequalizer of its kernel pair, $F \in \Fun(\ccc^{\op}, \Set)$ is a sheaf for $\ttt_{\Lambda}$ if and only if $F$ maps coequalizers of kernel pairs to equalizer diagrams.
\end{example}

Now we return to the setting of a small $\Z$-linear category $\ccc$ on which $\Lambda$ determines single $\Lambda$-topologies. We suppose moreover that the morphisms in $\Lambda$ have kernels.
Let $F: \ccc^{\op} \lra \Ab$ be a (possibly non-additive) functor. Let us write down the sheaf property as concretely as possible. For $\lambda: D \lra C$ in $\Lambda$, we obtain a diagram
\begin{equation}\label{pbplus}
\xymatrix{0 \ar[r] & K \ar[r]^{\kappa} \ar[dr]_{i_1} & D \ar[r]^{\lambda} \ar[r] & C \\
&& {K \oplus D} \ar[u]_-{\kappa + p_2} \ar[r]_-{p_2} & D \ar[u]_{\lambda} }
\end{equation}
in which the square is a kernel pair. The sheaf property for $F$ with respect to $\lambda$ requires that the sequence
\begin{equation} \label{pbplus2}
\xymatrix{0 \ar[r] & {F(C)} \ar[r]_-{F(\lambda)} & {F(D)} \ar[r]_-{F(\kappa + p_2) - F(p_2)} & {F(K \oplus D)}}\end{equation}
is exact. 

In the situation where $F: \ccc^{\op} \lra \Ab$ is additive, exactness of \eqref{pbplus2} clearly reduces to exactness of
\begin{equation}\label{addlex}
\xymatrix{0 \ar[r] & {F(C)} \ar[r]_-{F(\lambda)} & {F(D)} \ar[r]_-{F(\kappa)} & {F(K).}}
\end{equation}

\begin{definition}
An additive functor $F: \ccc^{\op} \lra \Ab$ is called \emph{$\Lambda$-left exact} if for every exact sequence
$$\xymatrix{0 \ar[r] & K \ar[r]_{\kappa} & D \ar[r]_{\lambda} & C}$$
with $\lambda \in \Lambda$
the sequence \eqref{addlex} is exact in $\Ab$.
\end{definition}

Let $\Lex_{\Lambda}(\ccc) \subseteq \Mod(\ccc)$ denote the full subcategory of $\Lambda$-left exact modules. We thus have:
$$\Sh_{\Lambda}^{\mathrm{add}}(\ccc) = \Lex_{\Lambda}(\ccc).$$
In a sense, the non-additive sheaf category $\Sh_{\Lambda}(\ccc)$ captures a kind of $\Lambda$-left exactness with additivity ``removed''.

\subsection{Exact categories}
Let $\ccc$ be an exact category in the sense of Quillen \cite{quillen2, keller9}. The exact structure on the additive category $\ccc$ is given by a collection of so called \emph{conflations}
\begin{equation}\label{confl} 
\xymatrix{K \ar[r]_{\kappa} & D \ar[r]_{\lambda} & C,}
\end{equation}
exact in the sense that $\kappa$ is a kernel of $\lambda$ and $\lambda$ is a cokernel of $\kappa$, satisfying some further axioms. Let $\Lambda$ be the collection of \emph{deflations}, i.e. morphisms $\lambda$ turning up in a conflation \eqref{confl}, and let $\Omega$ be the collection of \emph{inflations}, i.e. morphisms $\kappa$ turning up in a conflation \eqref{confl}. The further axioms of an exact category can be summarized as follows:
\begin{enumerate}
\item $\Lambda$ satisfies the conditions of Proposition \ref{propcond}.
\item $\Omega^{\op}$ satisfies the conditions of Proposition \ref{propcond} in $\ccc^{\op}$.
\end{enumerate}
Note that since $\kappa$ is required to be a cokernel of $\lambda$, the entire exact structure is in fact determined by the collection $\Lambda$. From now on, the exact structure of $\ccc$ being specified, we will drop the mention of $\Lambda$ from our notations and terminology. In this way we naturally recover the standard notions of weakly effaceable functors and left exact functors. 
It is well known (see \cite{keller9}) that the canonical embedding
$$\ccc \lra \Lex(\ccc): C \longmapsto \ccc(-,C)$$
is such that \eqref{confl} is a conflation in $\ccc$ if and only if 
$$\xymatrix{0 \ar[r] & K \ar[r]_{\kappa} & D \ar[r]_{\lambda} & C \ar[r] & 0}$$
is an exact sequence in $\Lex(\ccc)$.

Let $\Ind(\ccc) \subseteq \Mod(\ccc)$ denote the full subcategory of filtered colimits of $\ccc$-objects.
For a Grothendieck category $\ddd$, let $\mathsf{fp}(\ddd)$ denote the full subcategory of finitely presented objects. 

\begin{proposition}\label{epgen}
We have $\ccc \subseteq \mathsf{fp}(\Lex(\ccc))$ and $\Ind(\ccc) \subseteq \Lex(\ccc)$. The category $\Lex(\ccc)$ is locally finitely presented with $\ccc$ as a collection of finitely presented generators. In particular, $\mathsf{fp}(\Lex(\ccc))$ is the closure of $\ccc$ in $\Lex(\ccc)$ under finite colimits and every object in $\Lex(\ccc)$ is a filtered colimit of objects in $\mathsf{fp}(\Lex(\ccc))$. If $\ccc \cong \mathsf{fp}(\Lex(\ccc))$, then $\Ind(\ccc) =  \Lex(\ccc)$.
\end{proposition}

\begin{proof}
The objects $\ccc(-,C)$ are finitely presented in $\Mod(\ccc)$, so by Lemma \ref{filtsheaves} they are finitely presented in $\Lex(\ccc)$ as well. By the same lemma, filtered colimits of $\ccc$-objects in $\Mod(\ccc)$ remain left exact. The statements concerning local finite presentation are standard, in particular $F$ in $\Lex(\ccc)$ can be written as filtered colimit of $\mathsf{fp}(\Lex(\ccc))/F \lra \Lex(\ccc): (X \rightarrow F) \longmapsto X$. If $\ccc \cong \mathsf{fp}(\ccc)$, then again by Lemma \ref{filtsheaves}, this colimit can be computed in $\Mod(\ccc)$.
\end{proof}

\begin{examples}\label{examples}
\begin{enumerate}
\item Let $R$ be a ring. Let $\ccc_1 = \mathsf{free}(R)$ be the category of finitely generated free modules and $\ccc_2 = \mathsf{proj}(R)$ the category of finitely generated projective modules. Both subcategories of $\Mod(R)$ are closed under extensions (which are automatically split) whence inherit an exact structure from $\Mod(R)$. By Remark \ref{remsplit}, the topologies $\ttt_{\Lambda}$ and $\ttt_{\Lambda}^{\mathrm{add}}$ are trivial whence
$$\Lex(\ccc_i) = \Sh^{\mathrm{add}}(\ccc_i) = \Mod(\ccc_i) \cong \Mod(R)$$
and $$\Sh(\ccc_i) \cong \Fun(\ccc_i).$$
\item If $\ccc$ is a small abelian category with the canonical exact structure, then $\ccc$ is closed under finite colimits in $\Lex(\ccc)$ whence by Proposition \ref{epgen}, $\ccc \cong \mathsf{fp}(\Lex(\ccc))$ and $\Ind(\ccc) = \Lex(\ccc)$. Now let $\aaa$ be a locally coherent Grothendieck category, i.e. $\aaa$ is locally finitely presented and $\mathsf{fp}(\aaa) \subseteq \aaa$ is an abelian subcategory. Then by Proposition \ref{evgen}, $\aaa \cong  \Lex(\mathsf{fp}(\aaa)) \cong \Ind(\mathsf{fp}(\aaa))$. These facts are well known (see for example \cite{popescu}).

\item For a general Grothendieck category $\aaa$ the kernel of an epimorphism between finitely presented objects is not itself finitely presented, so $\mathsf{fp}(\aaa)$ does not inherit an exact structure from $\aaa$. By Proposition \ref{evgen}, it does however always inherit the single $\aaa$-epimorphism topology $\ttt$ for which
$$\aaa \cong \Sh^{\mathrm{add}}(\mathsf{fp}(\aaa), \ttt).$$
\item Clearly, the opposite category $\ccc^{\op}$ of an exact category becomes exact with $\Omega^{\op}$ playing the role of $\Lambda$. Thus, we obtain a canonical embedding
$$\ccc^{\op} \lra \Lex(\ccc^{\op}) = \Lex_{\Omega^{\op}}(\ccc^{\op}).$$
\end{enumerate}
\end{examples}
The definition of derived categories of abelian categories can be extended to exact categories (see \cite{neemanexact}, \cite{keller10}).

\begin{proposition}\label{propffd}
Consider the canonical embedding $\ccc \lra \Lex(\ccc)$. The canonical functor $D^-(\ccc) \lra D^-(\Lex(\ccc))$ is fully faithful.
\end{proposition}

\begin{proof}
By \cite[Theorem 12.1]{keller10}, this immediately follows from Lemma \ref{lemeff}.
\end{proof}

\begin{lemma}\label{lemeff}
Consider an epimorphism $F \lra C$ in $\Lex(\ccc)$ with $C \in \ccc$. There is a map $C' \lra F$ with $C' \in \ccc$ such that the composition $C' \lra F \lra C$ remains an epimorphism.
\end{lemma}

\begin{proof}
By Proposition \ref{epgen}, $C$ is finitely presented in $\Lex(\ccc)$, and $\Lex(\ccc)$ is a locally finitely presented category. Consider $f: F \lra C$ as stated. Writing $F = \colim_i M_i$ as a monofiltered colimit of its finitely generated subobjects, we have $C = \colim_i f(M_i)$. Since $C$ is finitely presented, the identity $1_C: C \lra \colim_i f(M_i)$ factors through some $f(M_j) \lra \colim_i f(M_i) = C$ which is then necessarily an isomorphism. Thus, we obtain an epimorphism $M = M_j \lra F \lra C$ with $M$ finitely generated. Now there is an epimorphism $\oplus_i C_i \lra M$ and since $M$ is finitely generated, an epimorphism $\oplus_{i = 1}^n C_i \lra M$. Finally, since $\ccc$ is additive, $C' = \oplus_{i = 1}^n C_i \in \ccc$ and we obtain the desired epimorphism $C' \lra M \lra F \lra C$.
\end{proof}

\subsection{Sheaves in two variables}\label{parbimodules}
If $\ccc$ is an exact category, then both $\ccc$ and $\ccc^{\op}$ are naturally endowed with single morphism topologies: the ``single deflation-topology'' on $\ccc$ and the ``single inflation-topology'' on $\ccc^{\op}$. Hence, it makes sense to consider bimodules and bifuncors over $\ccc$ that are sheaves in either of the two variables. In fact, we can develop everything for two possibly different sites $\aaa^{\op}$ and $\bbb$, which, for simplicity of exposition, we take to arise from exact categories.

Consider exact categories $\aaa$ and $\bbb$ and the bifunctor category $\Fun(\aaa^{\op} \times \bbb)$. We will introduce a list of subcategories $\Fun^{\ast}_{\star}(\aaa^{\op} \times \bbb)$, in which we consider functors that are additive in some of the arguments, and sheaves in some of the arguments.
We will indicate additivity by upper indices $\ast \in \{\varnothing, \triangleleft, \triangleright, \diamond \}$ (where $\ast = \varnothing$ means ``invisible index''): $\Fun^{\triangleleft}$ means additive in the first variable (i.e. all the $F(-,A)$ are additive), $\Fun^{\triangleright}$ means additive in the second variable (i.e. all the $F(B,-)$ are additive), $\Fun^{\diamond}$ means additive in both variables (i.e. $\Fun^{\diamond}(\aaa^{\op} \times \bbb) = \Mod(\aaa^{\op} \otimes \bbb)$), and $\Fun$ means additive in none of the variables. In the same way, we indicate sheaves by lower indices $\star \in \{\varnothing,  \triangleleft, \triangleright, \diamond \}$. So for example, $\Fun^{\triangleleft}_{\triangleright}(\aaa^{\op} \times \bbb)$ consists of functors $F$ for which every $F(-,A)$ is additive and every $F(B,-)$ is a sheaf.
We are interested in inclusions of the type
$$i: \Fun^{\ast}_{\star}(\aaa^{\op} \times \bbb) \lra \Fun^{\ast}(\aaa^{\op} \times \bbb)$$
where the ``additivity parameter'' is left unchanged, but we have inclusions of sheaves into presheaves in some of the arguments. Our first aim is to show that all these inclusions are localizations, just like
$$i_1: \Lex(\aaa) \lra \Mod(\aaa)$$
and $$i_2: \Sh(\aaa) \lra \Fun(\aaa)$$
in the one argument case. 
First note that $i_1$ and $i_2$ give rise to a number of localizations by looking at the induced $\Fun(\bbb, i_j)$ and $\Mod(\bbb, i_j)$, and dual versions of these. Also, it is immediate to write down the corresponding localizing Serre subcategories. For example, 
$$\Fun^{\triangleright}_{\triangleright}(\aaa^{\op} \times \bbb) \lra \Fun^{\triangleright}(\aaa^{\op} \times \bbb)$$
is realized as $\Fun(\bbb, i_1)$, and the corresponding localizing Serre subcategory consists of functors that are weakly effaceable in the second argument. In general, for $\ast \in \{\varnothing, \triangleleft, \triangleright, \diamond \}$ and $\star \in \{ \triangleleft, \triangleright \}$, we put
$$\www^{\ast}_{\star} = \www^{\ast}_{\star}(\aaa^{\op} \times \bbb) \subseteq \Fun^{\ast}_{\star}(\aaa^{\op} \times \bbb)$$
the subcategory of functors weakly effaceable in the argument designated by $\star$. For example, in the above example, the relevant category is $\www^{\triangleright}_{\triangleright}$.

Next we turn to the cases we haven't covered yet, namely the inclusions
$$i: \Fun^{\ast}_{\diamond}(\aaa^{\op} \times \bbb) \lra \Fun^{\ast}(\aaa^{\op} \times \bbb).$$
For localizing Serre subcategories $\sss_1$ and $\sss_2$ of an abelian category $\ccc$, we put $\sss_1 \ast \sss_2 = \{ C \in \ccc \, |\, \exists S_1 \in \sss_1, S_2 \in \sss_2, 0 \lra S_1 \lra C \lra S_2 \lra 0\}$. The subcategories are called \emph{compatible} \cite{verschoren1, vanoystaeyenverschoren} if $\sss_1 \ast \sss_2 = \sss_2 \ast \sss_1$. In this event $\sss_1 \ast \sss_2$ is the smallest localizing Serre subcategory containing $\sss_1$ and $\sss_2$ and
$$(\sss_1 \ast \sss_2)^{\perp} = \sss_1^{\perp} \cap \sss_2^{\perp}.$$

\begin{definition}
A module $W \in \Fun^{\ast}(\aaa^{\op} \times \bbb)$ is called \emph{weakly effaceable} if for every $\xi \in W(B,A)$, there exist a deflation $B' \lra B$ and an inflation $A \lra A'$ such that the induced $W(B,A) \lra W(B',A')$ maps $\xi$ to zero.
\end{definition}

\begin{proposition}\label{propwef}
The inclusion 
$$i: \Fun^{\ast}_{\diamond}(\aaa^{\op} \times \bbb) \lra \Fun^{\ast}(\aaa^{\op} \times \bbb)$$
is a localization with corresponding localizing subcategory 
$$\www^{\ast}_{\diamond} \doteqdot \www^{\ast}_{\triangleleft} \ast \www^{\ast}_{\triangleright}$$ consisting of all weakly effaceable bifunctors.
\end{proposition}

\begin{proof}
It suffices to show that $\www^{\ast}_{\triangleleft}$ and $\www^{\ast}_{\triangleright}$ are compatible and that $\www^{\ast}_{\triangleleft} \ast \www^{\ast}_{\triangleright}$ consists of the weakly effaceable bifunctors. Suppose we have an exact sequence $0 \lra W_1 \lra F \lra W_2 \lra 0$ in $\Fun^{\ast}(\aaa^{\op} \times \bbb)$ with $W_1 \in \www^{\ast}_{\triangleleft}$ and $W_2 \in \www^{\ast}_{\triangleright}$.
Consider $\xi \in F(B,A)$. Since $W_2$ is weakly effaceable in the second variable, there is an inflation $A \lra A'$ such that the image $\xi' \in F(B,A')$ of $\xi$ gets mapped to zero in $W_2(B,A')$. But then $\xi'$ is itself the image of some $\xi'' \in W_1(B,A)$. Now we can find a deflation $B' \lra B$ effaceing the image of $\xi'$ in $F(B',A')$. Clearly, this is independent of exchanging the roles of $W_1$ and $W_2$. Conversely, consider a weakly effaceable $W$. Define $W_1 \subseteq W$ by letting $W_1(B,A) \lra W(B,A)$ contain all elements $\xi$ that can be effaced in the first variable. It is readily seen that the quotient $W/W_1$ is weakly effaceable in the second variable.
\end{proof}

\section{Derived sheaf categories}\label{sec.2}

In this section we investigate the derived functors of the various inclusions of (bi)sheaf categories into (bi)functor categories of the previous section.

\subsection{Models of derived functors}
In this subsection we prove Lemma \ref{lemmodel} on the existence of dg models of certain derived functors.
Let $\ccc$ be a small exact category. Let $\bar{\ccc} \lra \ccc$ be a $k$-cofibrant dg resolution of the $k$-linear category $\ccc$. Consider $\iota: \bar{\ccc} \lra \ccc \lra  \Lex(\ccc) \lra C(\Lex(\ccc))$ as an object in the model category of dg functors $\mathsf{DgFun}(\bar{\ccc}, C(\Lex(\ccc^{\op})))$ of \cite[Proposition 5.1]{lowenvandenbergh2}. Then a fibrant replacement $\iota \lra E$ yields a dg functor $$E: \bar{\ccc} \lra \mathsf{Fib}(C(\Lex(\ccc)))$$ and fibrant replacements $C \lra E(C)$ natural in $C \in \bar{\ccc}$.

Now consider a left exact functor $F: \Lex(\ccc) \lra \Mod(k)$. It gives rise to a dg functor
$$F: \mathsf{Fib}(C(\Lex(\ccc))) \lra C(k).$$
The composition
$$FE: \bar{\ccc} \lra C(k)$$
induces a functor
$$RF: \ccc \cong H^{0}\bar{\ccc} \lra H^{0}C(k) \lra D(k)$$
which is a derived functor of $F$ (with restricted domain). Furthermore the natural functor
$$\mathsf{DgFun}(\bar{\ccc}, C(k)) \lra \Fun(\ccc, D(k))$$
clearly descends to a funcor
$$D(\bar{\ccc}^{\op}) \lra \Fun(\ccc, D(k)).$$
Next we will replace $FE$ by an honest dg functor $\ccc \lra C(k)$. To this end we note that $\bar{\ccc} \lra \ccc$ induces an equivalence of categories $D(\ccc^{\op}) \lra D(\bar{\ccc}^{\op})$. Let $\ovl{RF}: \ccc \lra C(k)$ be any representative in $C(\ccc^{\op}) = C(\Mod(\ccc^{\op}))$ of a pre-image of $FE$ under this equivalence.
Then  the induced functor $\ccc \lra H^0C(k) \lra D(k)$ is a derived functor of $F$ (with restricted domain).
We have thus proven:

\begin{lemma}\label{lemmodel}
Let $\ccc$ be a small exact category and $F: \Lex(\ccc) \lra \Mod(k)$ a left exact functor. There exists a complex $\ovl{RF} \in C(\Mod(\ccc^{\op}))$ such that the corresponding dg functor $\ovl{RF}: \ccc \lra C(k)$ induces a restriction $\ccc \lra H^0C(k) \lra D(k)$ of a derived functor of $F$.
\end{lemma}

\subsection{Derived localizations}
Next we will investigate the derived functors of the localizations of \S \ref{parbimodules}. The following general fact will be useful.
Consider a localization $i: \ccc \lra \ddd$ of Grothendieck categories with exact left adjoint $a: \ddd \lra \ccc$. We have a derived adjoint pair $Ri: D(\ccc) \lra D(\ddd)$ and $La = a: D(\ddd) \lra D(\ccc)$.

\begin{proposition}\label{fff}
The functor $Ri$ is fully faithful.
\end{proposition}

\begin{proof}
Endow $C(\ccc)$ and $C(\ddd)$ with the injective model structures for which cofibrations are pointwise monomorphisms and weak equivalences are quasi-isomorphisms. Since $a$ preserves both of these classes, by adjunction $i$ preserves fibrations and fibrant objects. For fibrant objects $E$ and $F$ in $C(\ccc)$ we have $\RHom(Ri(E), Ri(F)) = \RHom(i(E), i(F)) = \Hom(i(E), i(F)) = \Hom(E,F) = \RHom(E,F)$.\end{proof}

\subsection{The derived category of left exact modules}\label{parderlex}

Let $\ccc$ be a small exact category. We will now characterize the essential image of  $Ri: D^+(\Lex(\ccc)) \lra D^+(\Mod(\ccc))$.

\begin{definition}\label{cohom}
Let $\ccc$ be an exact category and $\ttt$ a triangulated category. A functor $F: \ccc \lra \ttt$ is called \emph{cohomological} if for every conflation $A \lra B \lra C$ in $\ccc$, the image under $F$ can be completed into a triangle $F(A) \lra F(B) \lra F(C) \lra F(A)[1]$ in $\ttt$.

A complex $K \in C(\Fun(\ccc))$ is called \emph{cohomological} if the induced functor $\ccc^{\op} \lra C(k) \lra D(k)$ is cohomological.
\end{definition}

\begin{examples}\label{excohom}
\begin{enumerate}
\item For a Grothendieck category $\aaa$, the natural functor $\aaa \lra D(\aaa)$ is cohomological.
\item If $\ccc' \lra \ccc$ is an exact functor between exact categories, $\ttt \lra \ttt'$ is a triangulated functor between triangulated categories, and $\ccc \lra \ttt$ is cohomological, then the composition $\ccc' \lra \ttt'$ is cohomological too.
\end{enumerate}
\end{examples}

\begin{proposition}\label{prop1}
Let $K \in C(\Mod(\ccc))$ be a bounded below complex. The following are equivalent:
\begin{enumerate}
\item $K \cong Ri (L)$ for some $L \in D(\Lex(\ccc))$;
\item $K \cong Ri (a(K))$;
\item $\RHom(W, K) = 0$ for every weakly effaceable $W$;
\item $K$ is cohomological.
\end{enumerate}
The implications from $(i)$ to $(j)$ with $i \leq j$ hold without the boundedness assumption on $K$.
\end{proposition}

\begin{proof}
The equivalence of (1) and (2) is obvious since, by Proposition \ref{fff}, $a Ri \cong 1$.

To see that (2) implies (3) we take $K$ as in (2) and $W$ weakly effaceable and we write $\RHom(W,K) = \RHom(W, Ri(a(K)) = \RHom(a(W),a(K)) = 0$ since $a(W) = 0$.

To show that (3) implies (4),  suppose that $K$ satisfies (3) and consider a conflation $A \lra B \lra C$. There is an associated exact sequence $0 \lra \aaa(-,A) \lra \aaa(-,B) \lra \aaa(-,C) \lra W \lra 0$ in $\Mod(\ccc)$ in which $W$ is weakly effaceable, proving (4).

To show that (4) implies (1), consider the adjunction morphism $K \lra Ri (a(K))$. We are to show that the cone $L$ is acyclic in $C(\Mod(\ccc))$. Then $L$ remains cohomological, and since $a(L)$ is acyclic in $C(\Lex(\ccc))$ the cohomology objects $H^i$ of $L$ are weakly effaceable. Since $L$ is bounded below, obviously there is an $n$ with $H^i = 0$ for all $i \leq n$. Let us prove that $H^i = 0$ implies $H^{i+1} = 0$. Consider $\xi \in H^{i+1}(C)$. There is a conflation $A \lra B \lra C$ such that $H^{i+1}(C) \lra H^{i+1}(B)$ maps $\xi$ to zero. The long exact cohomology sequence induced by the triangle $L(A) \lra L(B) \lra L(C) \lra L(A)[1]$ contains $\dots \lra H^i(A) \lra H^{i+1}(C) \lra H^{i+1}(B) \lra \dots$ so we conclude that $\xi = 0$. Consequently $H^{i+1} = 0$.
\end{proof}

We have the following partial counterpart for the inclusion 
$Ri': D^{+}(\Sh(\ccc)) \lra D^{+}(\Fun(\ccc))$.

\begin{proposition}\label{prop1bis}
Let $K \in C(\Fun(\ccc))$ be a bounded below complex. If $K$ is cohomological, then $K \cong Ri'(a'(K))$.
\end{proposition}

\begin{proof}
This is the same proof as for Proposition \ref{prop1}.
\end{proof}

\begin{remark}
Consider objects $A, B \in \Lex(\ccc)$. The fact that a cohomological complex $K \in C(\Lex(\ccc))$ resolving $B$ yields a cohomological complex $j'(K) \in \Sh(\ccc)$ resolving $j'(B)$ easily shows that the natural map $$\Ext^1_{\Lex(\ccc)}(A,B) \lra \Ext^1_{\Sh(\ccc)}(j'(A), j'(B))$$ is an isomorphism, a fact we already know from Proposition \ref{ext1}.
\end{remark}

\begin{corollary}
The functor $Ri: D^+(\Lex(\ccc)) \lra D^+(\Mod(\ccc))$ induces an equivalence $D^+(\Lex(\ccc)) \lra \widetilde{{D}^{+}}(\Lex(\ccc))$ where $\widetilde{{D}^{+}}(\Lex(\ccc)) \subseteq D^{+}(\Mod(\ccc))$ is the full subcategory of cohomological complexes.
\end{corollary}

Our main interest in Proposition \ref{prop1} stems from the following:

\begin{proposition}\label{mainint}
Consider a left exact functor $F': \Lex(\ccc)^{\op} \lra \Mod(k)$ with left exact restriction $F: \ccc^{\op} \lra \Mod(k)$ in $\Lex(\ccc)$.
Let $Ri(F)$ be the image of $F$ under $$Ri: D(\Lex(\ccc)) \lra D(\Mod(\ccc)).$$ Any representative $\ccc^{\op} \lra C(k)$ of $Ri(F)$ induces a functor $\ccc^{\op} \lra C(k) \lra  D(k)$ which is a restriction to $\ccc^{\op}$ of a derived functor of $F'$.
\end{proposition}

\begin{proof}
Apart from our standard universe $\mathsf{U}$, take a larger universe $\mathsf{V}$ such that $\ddd = \Lex(\ccc)^{\op}$ is $\mathsf{V}$-small. Since $\ddd$ is abelian, we can extend $F'': \ddd \lra \Mod(k) \lra \mathsf{V}$-$\Mod(k)$ to a left exact functor
$$\hat{F}'': \mathsf{V}\text{-}\Lex(\ddd) \lra \mathsf{V}\text{-}\Mod(k): \colim_i D_i \longmapsto \colim_i F''(D_i).$$
Now we can apply Lemma \ref{lemmodel} to $\hat{F}''$. We thus obtain $\ovl{RF''} \in C(\mathsf{V}\text{-}\Mod(\ddd^{\op}))$ inducing a restriction $RF'': \ddd \lra \mathsf{V}\text{-}C(k) \lra \mathsf{V}\text{-}D(k)$ of a derived functor of $\hat{F}''$, which is itself a derived functor of $F''$. If we restrict $\ovl{RF''}$ to $\ovl{RF} \in C(\mathsf{V}\text{-} \Mod(\ccc))$, then this complex is such that the induced functor $\ccc^{\op} \lra \mathsf{V}\text{-}C(k) \lra \mathsf{V}\text{-}D(k)$ is a restriction of a derived functor of $F''$. By Examples \ref{excohom}, $\ovl{RF}$ is cohomological, whence, by Proposition \ref{prop1}, $\ovl{RF} \cong Ri(a(\ovl{RF}))$, where we consider $i: \mathsf{V}\text{-} \Lex(\ccc) \lra \mathsf{V}\text{-} \Mod(\ccc)$ and its left adjoint $a$.
Clearly the $n\text{-}th$ cohomology object of $\ovl{RF}$ corresponds to $H^n\ovl{RF}: \ccc^{\op} \lra \mathsf{V}\text{-} C(k) \lra \mathsf{V}\text{-} \Mod(k)$, which is the restriction to $\ccc^{\op}$ of the $n\text{-}th$ derived functor $R^nF'': \Lex(\ccc)^{\op} \lra \mathsf{V}\text{-} \Mod(k)$ of $F''$. Now $R^nF''$ is effaceable, so for $C \in \ccc$ there is an epimorphism $u: X \lra C$ in $\Lex(\ccc)$ such that $R^nF''(u) = 0$. By Lemma \ref{lemeff} there is a further morphism $v: C' \lra X$ in $\Lex(\ccc)$ with $C' \in \ccc$ such that $uv: C' \lra C$ remains an epimorphism. In particular, $H^n(\ovl{RF}) \in \mathsf{V}\text{-} \Mod(\ccc)$ is weakly effaceable for $n > 0$ and $H^0(\ovl{RF}) = F: \ccc^{\op} \lra \mathsf{V}\text{-} \Mod(k)$. Consequently, $a(\ovl{RF}) = F$. Moreover, since in fact $F: \ccc^{\op} \lra \mathsf{U} \text{-} \Mod$, we have $\ovl{RF} \cong Ri(F) \in C(\mathsf{U} \text{-} \Mod(\ccc))$.
\end{proof} 

\begin{remark}
Any left exact functor $F: \ccc^{\op} \lra \Mod(k)$ in $\Lex(\ccc)$ has a left exact extension
$$F' = \Lex(\ccc)(-, F): \Lex(\ccc)^{\op} \lra \Mod(k).$$
Since $Ri(F)$ is obtained by replacing $F$ by an injective resolution in $\Lex(\ccc)$, Proposition \ref{mainint} is a kind of balancedness result.
\end{remark}

\subsection{Localization in one of several variables}
Let $\AAA$ be a small $k$-cofibrant dg category and $i: \LLL \lra \ddd$, $a: \ddd \lra \LLL$ a localization between Grothendieck categories. Consider the induced localization
$$i\circ -: \mathsf{DgFun}(\AAA, C(\LLL)) \lra \mathsf{DgFun}(\AAA, C(\ddd)),$$
$$a\circ -: \mathsf{DgFun}(\AAA, C(\ddd)) \lra \mathsf{DgFun}(\AAA, C(\LLL))$$
where the involved categories are endowed with the model structure of \cite[Proposition 5.1]{lowenvandenbergh2}.

\begin{lemma}\label{lemEF}
If $F \in \mathsf{DgFun}(\AAA, C(\LLL))$ is such that for every $A \in \AAA$, $F(A)$ is fibrant, then $F$ is $(i \circ -)$-acyclic, i.e $R(i \circ -)(F) \cong iF$ in $\mathrm{Ho}(\mathsf{DgFun}(\AAA, C(\LLL)))$.
\end{lemma}

\begin{proof}
Take a fibrant resolution $F \lra E$ in $\mathsf{DgFun}(\AAA, C(\LLL))$. Since $\AAA$ is $k$-cofibrant, for every $A \in \AAA$, $E(A)$ is fibrant. Consequently, every $F(A) \lra E(A)$ is a weak equivalence between fibrant objects, whence a homotopy equivalence. Since $i(F(A)) \lra i(E(A))$ remains a homotopy equivalence, $iF \lra iE$ is a weak equivalence as desired.
\end{proof}

\subsection{Sheaves in one of several variables}
As soon as we want to extend the results of the previous subsections to bimodules, flatness over the ground ring comes into play. The reason for this is that in the absence of flatness, injective resolutions of bimodules do not yield injective resolutions in individual variables. More precisely, we have the following situation. Let $\ccc$ be a small exact category and $\AAA$ a small $k$-linear category, and consider the category $\Mod(\AAA, \Mod(\ccc)) \cong \Mod(\AAA^{\op} \otimes \ccc).$ The localization $\Lex(\ccc) \lra \Mod(\ccc)$ gives rise to a localization
$$i_{\ccc}: \Mod(\AAA, \Lex(\ccc)) \lra \Mod(\AAA, \Mod(\ccc)).$$

\begin{lemma}\label{lemflat}
For every $A \in \AAA$, the projection $\mathrm{ev}_A: \Mod(\AAA, \Lex(\ccc)) \lra \Lex(\ccc): F \longmapsto F(A)$ has a left adjoint
given by $M \longmapsto (A' \longmapsto \AAA(A,A') \otimes_k M)$.
If $\AAA$ has $k$-flat homsets, then this adjoint is exact, and $\mathrm{ev}_A$ preserves injectives.
\end{lemma} 

\begin{proof}
This is clear.
\end{proof}

Let $\Lex(\Lex(\ccc))$ denote the category of left exact additive functors $\Lex(\ccc)^{\op} \lra \Mod(k)$. The exact inclusion $s: \ccc \lra \Lex(\ccc)$ induces a restriction $\pi: \Lex(\Lex(\ccc)) \lra \Lex(\ccc): G \longmapsto Gs$
and the inclusion functor
$\iota: \Lex(\ccc) \lra \Lex(\Lex(\ccc)): F \longmapsto \Lex(\ccc)(-,F)$
satisfies $\pi \iota = 1_{\Lex(\ccc)}$.

Let $F': \AAA \lra \Lex(\Lex(\ccc))$ be an additive functor with restriction $F = \pi F': \AAA \lra \Lex(\ccc)$. The following result extends Proposition \ref{mainint} to modules left exact in one of several variables.

\begin{proposition}\label{mainint2}
Let $\AAA$, $\ccc$, $F'$ and $F$ be as above and let $Ri_{\ccc}(F)$ be the image of $F$ under
$$Ri_{\ccc}: D(\Mod(\AAA, \Lex(\ccc))) \lra D(\Mod(\AAA, \Mod(\ccc))).$$
If $\AAA$ has $k$-flat homsets, then for any $K \in C(\Mod(\AAA, \Mod(\ccc)))$ representing $Ri_{\ccc}(F)$ and for any $A \in \AAA$, $K(A) \in C(\Mod(\ccc))$ induces a functor $\ccc^{\op} \lra C(k) \lra D(k)$ which is a restriction to $\ccc^{\op}$ of a derived functor of $F'(A): \Lex(\ccc)^{\op} \lra \Mod(k)$.\end{proposition}

\begin{proof}
Let $F \lra E$ be an injective resolution of $F \in \Mod(\AAA, \Lex(\ccc))$. Then for every $A \in \AAA$, $F(A) \lra E(A)$ is an injective resolution in $\Lex(\ccc)$ by Lemma \ref{lemflat}.
Consequently, for the inclusion $i: \Lex(\ccc) \lra \Mod(\ccc)$, we have 
$Ri(F(A)) = i(E(A)) = i_{\ccc}(E)(A) = Ri_{\ccc}(F)(A)$ in $D(\Mod(\ccc))$ hence the result follows from  Proposition \ref{mainint}.
\end{proof}

If, in the first argument, we consider functors rather than modules, the flatness issue goes away. We are interested in the following application. Let $\bbb$ and $\aaa$ be small exact categories and let $F': \aaa \lra \Lex(\Lex(\bbb))$ be a possibly non-additive functor with restriction $F = \pi F': \aaa \lra \Lex(\bbb)$. Consider the inclusion $i_{\bbb}: \Fun(\aaa, \Lex(\bbb)) \lra \Fun(\aaa, \Mod(\bbb)).$

\begin{corollary}\label{maincor}
Let $\aaa$, $\bbb$, $F'$ and $F$ be as above and let $Ri_{\bbb}(F)$ be the image of $F$ under
$$Ri_{\bbb}: D(\Fun(\aaa, \Lex(\bbb))) \lra D(\Fun(\aaa, \Mod(\bbb))).$$
For any $K \in C(\Fun(\aaa, \Mod(\bbb)))$ representing $Ri_{\bbb}(F)$ and for any $A \in \aaa$, $K(A) \in C(\Mod(\bbb))$ induces a functor $\bbb^{\op} \lra C(k) \lra D(k)$ which is a restriction to $\bbb^{\op}$ of a derived functor of $F'(A): \Lex(\bbb)^{\op} \lra \Mod(k)$.
\end{corollary}

\begin{proof}
This immediately follows from Proposition \ref{mainint2} by putting $\ccc = \bbb$ and $\AAA = \Z \aaa$, the free $\Z$-linear category on $\aaa$ (having $\Ob(\Z \aaa) = \Ob(\aaa)$ and $(\Z \aaa)(A,A') = \Z (\aaa(A,A'))$, the free abelian group on $\aaa(A,A')$), and noting that $\Fun(\aaa, \Lex(\bbb)) \cong\Mod(\Z \aaa, \Lex(\bbb))$.
\end{proof}

\subsection{Sheaves in two variables} \label{sub2var}

In this section we
consider sheaves in both variables. We start with a version of Proposition \ref{prop1}.
For small exact categories $\aaa$ and $\bbb$, consider the inclusions
$$i: \Fun_{\diamond}^{\ast}(\aaa^{\op} \times \bbb) \lra \Fun^{\ast}(\aaa^{\op} \times \bbb)$$
for $\ast \in \{ \varnothing, \triangleleft, \triangleright, \diamond\}$, along with the derived functors
$$Ri: D(\Fun_{\diamond}^{\ast}(\aaa^{\op} \times \bbb)) \lra D(\Fun^{\ast}(\aaa^{\op} \times \bbb)).$$
As usual, the left adjoints of $i$ and $Ri$ are denoted by $a$. 
For modules $F \in \Mod(\bbb)$ and $G \in \Mod(\aaa^{\op})$, $F \boxtimes G \in \Mod(\aaa^{\op} \otimes \bbb)$ denotes the bimodule with $(F \boxtimes G)(B,A) = F(B) \otimes G(A)$.

\begin{proposition}\label{bimodprop}
For $K \in C(\Fun^{\ast}(\aaa^{\op} \times \bbb))$, consider the following properties:
\begin{enumerate}
\item $K \cong Ri (L)$ for some $L \in D(\Fun^{\ast}_{\diamond}(\aaa^{\op} \times \bbb))$;
\item $K \cong Ri (a(K))$;
\item $\RHom(W, K) = 0$ for every weakly effaceable $W$;
\item $K$ is cohomological in both variables.
\item $K$ is cohomological in the first variable (i.e. for every $A \in \aaa$, the complex $K(-,A) \in C(Fun(\bbb))$ is cohomological).
\end{enumerate}

The following facts hold true:
\begin{enumerate}
\item[(i)] (1) and (2) are equivalent and (1) implies (3).
\item[(ii)] If $K$ is bounded below, then (4) implies (1).
\item[(iii)] If $k = \Z$ and $\ast = \triangleleft$, then (3) implies (5).
\item[(iv)] If $k$ is a field and $\ast = \diamond$, then (3) implies (4).
\end{enumerate}
\end{proposition}

\begin{proof}
(i) This is proven like in Proposition \ref{prop1}. (ii)
Suppose that $K$ is bounded below and that (4) holds. To prove that (4) implies (1), as in the proof of Proposition \ref{prop1} it is sufficient to show that if $K$ is cohomological in both variables and has weakly effaceable cohomology objects $H^i$, then $H^i = 0$ implies $H^{i+1} = 0$. Consider $\xi \in H^{i+1}(B,A)$. Take conflations $A \lra A' \lra A''$ and $B'' \lra B' \lra B$ such that $H^{i+1}(B,A) \lra H^{i+1}(B',A')$ maps $\xi$ to zero. From the diagram 
$$\xymatrix{ && {H^i(B'',A')} \ar[d] \\ {H^i(B,A'')} \ar[r] & {H^{i+1}(B,A)} \ar[d] \ar[r] & {H^{i+1}(B,A')} \ar[d] \\
& {H^{i+1}(B',A)} \ar[r] & {H^{i+1}(B',A')}}$$
with exact middle row and last column we deduce that $\xi = 0$. Consequently $H^{i+1} = 0$. 

We now give the proof of (iii), the proof of (iv) is similar.
Let $k = \Z$ and suppose $K$ satisfies (3). Consider a conflation $B' \lra B \lra B''$ in $\bbb$ and the asociated exact sequence $0 \lra \bbb(-,B') \lra \bbb(-,B) \lra \bbb(-,B'') \lra W \lra 0$ with $W$ weakly effaceable in $\Mod(\bbb)$. For $A \in \aaa$, the sequence $0 \lra \bbb(-,B') \boxtimes \aaa(A,-) \lra \bbb(-,B) \boxtimes \aaa(A,-) \lra \bbb(-,B'') \boxtimes \aaa(A,-) \lra W \boxtimes \aaa(A,-) \lra 0$ remains exact in $\Mod(\aaa^{\op} \otimes \Z \bbb)$. An element $\sum_{i = 1}^n w_i \otimes f_i \in W(Y) \otimes \Z \aaa(A,X)$ can be effaced by composing finitely many $\bbb$-deflations, so $W \boxtimes \aaa(A,-)$ is weakly effaceable in the first variable in  $\Fun^{\triangleleft}(\aaa^{\op} \times \bbb)$. Finally, since $\bbb(-,B) \boxtimes \Z \aaa(A,-) = ((\Z \aaa)^{op} \otimes \bbb)(-, (B,A))$, we obtain the desired triangle by considering $\RHom(-, K)$.
\end{proof}

\begin{corollary}
Suppose $k$ is a field.
The functor $Ri: D^+(\Fun^{\diamond}_{\diamond}(\aaa^{\op} \times \bbb)) \lra D^+(\Fun^{\diamond}(\aaa^{\op} \times \bbb))$ induces an equivalence $D^+(\Fun^{\diamond}_{\diamond}(\aaa^{\op} \times \bbb)) \lra \widetilde{D^+}(\Fun^{\diamond}_{\diamond}(\aaa^{\op} \times \bbb))$ where $\widetilde{D^+}(\Fun^{\diamond}_{\diamond}(\aaa^{\op} \times \bbb)) \subseteq D^+(\Fun^{\diamond}(\aaa^{\op} \times \bbb))$ is the full subcategory of complexes that are cohomological in both variables.
\end{corollary}

For small exact categories $\aaa$ and $\bbb$, consider the inclusions
$$i: \Fun^{\triangleleft}_{\diamond}(\aaa^{\op} \times \bbb) \lra \Fun^{\triangleleft}(\aaa^{\op} \times \bbb)$$
and
$$i_{\bbb}: \Fun^{\triangleleft}_{\triangleleft}(\aaa^{\op} \times \bbb) \lra \Fun^{\triangleleft}(\aaa^{\op} \times \bbb)$$
which has an equivalent incarnation:
$$i_{\bbb}: \Fun(\aaa, \Lex(\bbb)) \lra \Fun(\aaa, \Mod(\bbb)).$$
Consider $F \in \Fun^{\triangleleft}_{\diamond}(\aaa^{\op} \times \bbb)$ along with its natural extension
$$F': \aaa \lra \Lex(\bbb) \lra \Lex(\Lex(\bbb): A \longmapsto \Lex(\bbb)(-, F(A)).$$
Let $F \lra E$ be an injective resolution in $\Fun(\aaa, \Lex(\bbb))$. We have $Ri_{\bbb}(F) = i_{\bbb}E = E$ and for every $A \in \aaa$, $F(-,A) \lra E(-,A)$ is an injective resolution in $\Lex(\bbb)$. 
According to Corollary \ref{maincor}, $E(-,A): \bbb^{\op} \lra C(k)$ induces a restriction of a derived functor of $\Lex(\bbb)(-, F(A))$, and $E(-,A)$ is itself a restriction of $\Hom_{\Lex(\bbb)}(-, E(-,A))$.

\begin{proposition}\label{key}
Let $\aaa$, $\bbb$ and $F$ be as above. Suppose $F: \aaa \lra \Lex(\bbb)$ is exact (i.e. maps conflations to short exact sequences). Then $Ri_{\bbb}(F)$ is cohomological in both variables and $Ri_{\bbb}(F) \cong Ri(F)$.
\end{proposition}

\begin{proof}
Let $Ri_{\bbb}(F) = E$ as above. By the above discussion, for $A \in \aaa$, $E(-,A): \bbb^{\op} \lra C(k)$ is cohomological, $H^0E(-,A) = F(-,A)$ and the other cohomology object are weakly effaceable in $\Mod(\bbb)$. In particular, $H^0E = F$ and the higher cohomology objects are weakly effaceable, whence $a(E) = F$. Now fix $B \in \bbb$ and consider $E(B,-): \aaa \lra C(k)$. Let us show that this functor is cohomological. Let $A' \lra A \lra A''$ be a conflation in $\aaa$. By assumption, $0 \lra F(A') \lra F(A) \lra F(A'') \lra 0$ is a short exact sequence in $\Lex(\bbb)$. Now by naturality of $F \lra E$ we obtain a commutative diagram
$$\xymatrix{{F(A')} \ar[r] \ar[d] & {F(A)}
\ar[r] \ar[d] & {F(A'')} \ar[d] \\ {E(A')} \ar[r] & {E(A)} \ar[r] & {E(A'')}}$$
 in $C(\Lex(\bbb))$ in which the vertical arrows are quasi-isomorphisms. As a consequence, the lower row can be completed into a triangle in $\mathsf{Fib}(C(\Lex(\bbb)))$. The functor $\Lex(\bbb)(\bbb(-,B), -): \mathsf{Fib}(C(\Lex(\bbb))) \lra C(k)$ maps this triangle to a triangle in $D(k)$ as desired. Finally by Proposition \ref{bimodprop}, we conclude that $E \cong Ri(F)$.
\end{proof}

In the remainder of this subsection, let $k$ be a field. For small exact $k$-linear categories $\aaa$ and $\bbb$, 
consider the inclusions
$$i: \Fun^{\diamond}_{\diamond}(\aaa^{\op} \times \bbb) \lra \Fun^{\diamond}(\aaa^{\op} \times \bbb),$$
$$i_{\aaa}: \Fun^{\diamond}_{\triangleright}(\aaa^{\op} \times \bbb) \lra \Fun^{\diamond}(\aaa^{\op} \times \bbb)$$
which has an equivalent incarnation:
$$i_{\aaa}: \Mod(\bbb^{\op}, \Lex(\aaa^{\op})) \lra \Mod(\bbb^{\op}, \Mod(\aaa^{\op}))$$
and 
$$i_{\bbb}: \Fun^{\diamond}_{\triangleleft}(\aaa^{\op} \times \bbb) \lra \Fun^{\diamond}(\aaa^{\op} \times \bbb)$$
which has an equivalent incarnation:
$$i_{\bbb}: \Mod(\aaa, \Lex(\bbb)) \lra \Mod(\aaa, \Mod(\bbb)).$$
Consider $F \in \Fun^{\diamond}_{\diamond}(\aaa^{\op} \times \bbb)$.

\begin{proposition}\label{1naar2}
If $F: \bbb^{\op} \lra \Lex(\aaa^{\op})$ is exact, then $Ri_{\aaa}(F)$ is cohomological in both variables and $Ri_{\aaa}(F) \cong Ri(F)$. If $F: \aaa \lra \Lex(\bbb)$ is exact, then $Ri_{\bbb}(F)$ is cohomological in both variables and $Ri_{\bbb}(F) \cong Ri(F)$.
\end{proposition}

\begin{proof}
Similar to the proof of Proposition \ref{key}.
\end{proof}

\section{Cohomology of exact categories}\label{sec.3}

In this section we discuss a number of different cohomology expressions for exact categories and more generally for linear sites. We start with expressions ``of Hochschild type''. Our main results are over a field. We relate the cohomology of a Grothendieck category $\ddd$ of \cite{lowenvandenbergh2} to $\Ext$'s in the large additive functor category $\Add(\ddd, \ddd)$ (Theorem \ref{maingroth}). For a small exact category $\ccc$, the cohomology of \cite{keller6} corresponds to the cohomology of the Grothendieck category $\Lex(\ccc)$, similar to the situation for abelian categories in \cite{lowenvandenbergh2} (Theorem \ref{kelcomp}). We show that this cohomology can also be expressed as $\Ext$'s in the category $\Fun^{\diamond}_{\diamond}(\ccc^{\op} \times \ccc)$ of bimodules that are sheaves (in other words, left exact) in both variables (Theorem \ref{bilex}). This expression originated from \cite{kaledin1}. For module categories, some of these Hochschild expressions bear resemblance to an incarnation of Mac Lane cohomology discovered in \cite{jibladzepirashvili}. Inspired by this, we define Mac Lane cohomology for linear sites (\S \ref{MLadd}). Finally, we show that for an exact category $\ccc$ this cohomology can also be expressed as $\Ext$'s in the category $\Fun^{\triangleleft}_{\diamond}(\ccc^{\op} \times \ccc)$ of bifunctors that are additive in the first variable and sheaves in both variables (Theorem \ref{MLex}).

\subsection{Hochschild-Shukla cohomology of dg categories}
Let $k$ be a commutative ring. Let $\AAA$ be a $k$-linear dg category and $M$ an $\AAA$-bimodule. Recall that the Hochschild complex $\CC_{\mathrm{hoch}}(\AAA, M)$ of $\AAA$ with values in $M$ is the product total complex of the double complex with $p$-th column
$$\prod_{A_0, \dots, A_p} \Hom_k(\AAA(A_{p-1}, A_p) \otimes_k \dots \otimes_k \AAA(A_0, A_1), M(A_0, A_p))$$
and the usual Hochschild differential. The Hochschild complex of $\AAA$ is $\CC_{\mathrm{hoch}}(\AAA) = \CC_{\mathrm{hoch}}(\AAA, \AAA)$. If $\AAA$ is $k$-cofibrant, then
$$\CC_{\mathrm{hoch}}(\AAA, M) \cong \RHom_{\AAA^{\op} \otimes \AAA}(\AAA, M)$$
in $D(k)$.

For $\AAA$ arbitrary, the Shukla complex of $\AAA$ is by definition the Hochschild complex of a $k$-cofibrant dg resolution $\bar{\AAA} \lra \AAA$, i.e.
$$\CC_{\mathrm{sh}}(\AAA, M) = \CC_{\mathrm{hoch}}(\bar{\AAA}, M).$$

\subsection{Hochschild-Shukla cohomology of Grothendieck categories}\label{subgroth}
In \cite{lowenvandenbergh2}, Hochschild-Shukla cohomology was defined for abelian categories. For a Grothendieck category, a convenient definition is
$$\CC_{\mathrm{gro}}(\ddd) = \CC_{\mathrm{sh}}(\inj(\ddd))$$
where $\inj(\ddd)$ is the linear category of injectives in $\ddd$. 
Now let $(\UUU, \ttt)$ be an additive site with additive sheaf category $\Sh(\UUU)$ and canonical map $u: \UUU \lra \Sh(\UUU)$. For every $U \in \UUU$, choose an injective resolution $u(U) \lra E_U$ and let $\UUU_{\dg} \subseteq C(\Sh(\UUU))$ be the full dg subcategory consisting of the $E_U$. It is proven in \cite{lowenvandenbergh2} that
$$\CC_{\mathrm{gro}}(\Sh(\UUU)) \cong \CC_{\sh}(\UUU_{\dg}).$$
We finally recall the following more technical result \cite[Lemma 5.4.2]{lowenvandenbergh2}, which will be crucial for us. Let $r: \bar{\UUU} \lra \UUU$ be a $k$-cofibrant resolution and take a fibrant replacement $ur \lra E$ in the model category $\mathsf{DgFun}(\bar{\UUU}, C(\Sh(\UUU)))$ of \cite[Proposition 5.1]{lowenvandenbergh2}. Then $E$ naturally defines a $\bar{\UUU}$-$\bar{\UUU}$-bimodule by $E(U,V) = E(V)(U) = \Hom_{\Sh(\UUU)}(ur(U), E(V))$ and we have
\begin{equation}\label{eqkey}
\CC_{\gro}(\Sh(\UUU)) \cong \CC_{\sh}(\bar{\UUU}, E).
\end{equation}

In the remainder of this subsection, let $k$ be a field. Consider the localization
$$i \circ -: \Mod(\UUU, \Sh(\UUU)) \lra \Mod(\UUU, \Mod(\UUU)),$$
$$a \circ -: \Mod(\UUU, \Mod(\UUU)) \lra \Mod(\UUU, \Sh(\UUU))$$
induced by $i: \Sh(\UUU) \lra \Mod(\UUU)$, $a: \Mod(\UUU) \lra \Sh(\UUU)$.

\begin{proposition}\label{newgroth}
We have:
$$\CC_{\gro}(\Sh(\UUU))  \cong \CC_{\hoch}(\UUU, R(i\circ -)(u))  \cong \RHom_{\Mod(\UUU, \Sh(\UUU))}(u,u).$$
Furthermore, for every natural transformation $u \lra F$ in $C(\Mod(\UUU, \Sh(\UUU)))$ for which every $u(U) \lra F(U)$ is an injective resolution, we have:
$$\CC_{\gro}(\Sh(\UUU)) \cong \CC_{\hoch}(\UUU, F).$$
\end{proposition}

\begin{proof}
Since we are over a field, we can take $\bar{\UUU} = \UUU$ and $u \lra E$ an injective resolution of $u$ in $\Mod(\UUU, \Sh(\UUU))$.  By construction $R(i \circ -)(u) = iE$ and hence
$$\begin{aligned}
\CC_{\gro}(\Sh(\UUU)) & \cong \CC_{\hoch}(\UUU, iE)\\
& \cong \RHom_{\UUU^{\op} \otimes \UUU}(I, R(i \circ -)(u))\\
& \cong \RHom_{\Mod(\UUU, \Sh(\UUU))}(u,u).
\end{aligned}$$
Furthermore, by Lemma \ref{lemEF} we have $R(i \circ -)(u) \cong iF$.
\end{proof}

\subsection{Hochschild-Shukla cohomology of exact categories}
Let $k$ be a commutative ring. Let $\ccc$ be a small exact category. In this section we discuss some definitions of Hochschild-Shukla cohomology of $\ccc$.

The first definition is due to Keller \cite{keller6}. Let $C_{\mathrm{dg}}^b(\ccc)$ be the dg category of bounded complexes of $\ccc$-objects, and $Ac_{\mathrm{dg}}^b(\ccc)$ its full dg subcategory of acyclic complexes. Then for the dg quotient $D^b_{\mathrm{dg}}(\ccc) = C_{\mathrm{dg}}^b(\ccc)/Ac_{\mathrm{dg}}^b(\ccc)$:
$$\CC_{\mathrm{ex}}(\ccc) = \CC_{\mathrm{sh}}(D^b_{\mathrm{dg}}(\ccc)).$$
In \cite{lowenvandenbergh2}, the authors defined Hochschild-Shukla cohomology of abelian categories. This definition has the following generalization to exact categories:
$$\CC_{\mathrm{ex'}}(\ccc) = \CC_{\gro}(\Lex(\ccc)) = \CC_{\mathrm{sh}}(\inj(\Lex(\ccc)).$$

Using Proposition \ref{propffd} it is easily seen (see \cite[Lemma 6.3]{lowenvandenbergh2}) that a concrete model for $D^b_{\mathrm{dg}}(\ccc)$ is given by the full subcategory of $C_{\mathrm{dg}}(\Lex(\ccc))$ of bounded below complexes of injectives with bounded cohomology in $\ccc$. We also introduce the full subcategory $\ccc_{\mathrm{dg}} \subseteq C_{\mathrm{dg}}(\Lex(\ccc))$ of positively graded complexes of injectives whose only cohomology is in degree zero and in $\ccc$.

The following is proven in exactly the same way as \cite[Theorem 6.2]{lowenvandenbergh2}:

\begin{theorem}\label{kelcomp}
There are quasi-isomorphisms
$$\CC_{\mathrm{ex'}}(\ccc) \cong \CC_{\mathrm{sh}}(\ccc_{\mathrm{dg}}) \cong \CC_{\mathrm{ex}}(\ccc).$$
\end{theorem}

Shukla cohomology of an exact category interpolates between Shukla cohomology of a $k$-linear category and Shukla cohomology of an abelian category. Of course, an arbitrary $k$-linear category is not exact since it is not additive, but this can easily be remedied by adding finite biproducts.

\begin{proposition}\label{linex}
Let $\AAA$ be a $k$-linear category and $\mathsf{free}(\AAA)$ the exact category of finitely generated free $\AAA$-modules with split exact conflations. We have:
$$\CC_{\mathrm{ex}}(\mathsf{free}(\AAA)) \cong \CC_{\mathrm{sh}}(\AAA).$$
\end{proposition}

\begin{proof}
We have $\Lex(\mathsf{free}(\AAA)) \cong \Mod(\mathsf{free}(\AAA)) \cong \Mod(\AAA)$ (see Remark \ref{remsplit}). Hence it follows from \cite{lowenvandenbergh2} that $\CC_{\mathrm{ex}'}(\mathsf{free}(\AAA)) \cong \CC_{\mathrm{sh}}(\AAA)$.
\end{proof}

\subsection{Hochschild cohomology and (bi)sheaf categories}\label{hochbisheaf}
Let $k$ be a field and $\ccc$ a small exact $k$-linear category. Let $\iota: \ccc \lra \Lex(\ccc)$ be the canonical embedding. The results of the previous subsections yield:

\begin{proposition}\label{exabbas}
We have:
$$\CC_{\ex}(\ccc) \cong \RHom_{\Mod(\ccc, \Lex(\ccc))}(\iota, \iota).$$
\end{proposition}

\begin{proof}
This is an application of Proposition \ref{newgroth} to $u = \iota: \ccc \lra \Lex(\ccc)$. 
\end{proof}

Let $I$ denote the identity $\ccc$-bimodule with $I(C',C) = \ccc(C',C)$.
Using the results of \S \ref{sub2var}, we obtain the following symmetric abelian expression, which also appeared in \cite{kaledin1}:

\begin{theorem}\label{bilex}
We have:
$$\CC_{\ex}(\ccc) \cong  \RHom_{\Fun^{\diamond}_{\diamond}(\ccc^{\op} \times \ccc)}(I,I).$$
\end{theorem}

\begin{proof}
Consider $i_1: \Fun^{\diamond}_{\triangleleft}(\ccc^{\op} \times \ccc) \lra \Fun^{\diamond}_{\diamond}(\ccc^{\op} \times \ccc)$, which is isomorphic to $i-: \Mod(\ccc, \Lex(\ccc)) \lra \Mod(\ccc, \Mod(\ccc))$.
Proposition \ref{exabbas} translates into $\CC_{\ex}(\ccc) \cong \RHom_{\Fun^{\diamond}_{\triangleleft}(\ccc^{\op} \times \ccc)}(I,I)$. From Proposition \ref{1naar2} we further obtain $Ri_1(I) \cong Rj(I)$ for $j: \Fun^{\diamond}_{\diamond}(\ccc^{\op} \times \ccc) \lra \Fun^{\diamond}(\ccc^{\op} \times \ccc)$, so 
$$\CC_{\ex}(\ccc) \cong \RHom_{\Fun^{\diamond}}(I, Rj(I)) \cong \RHom_{\Fun^{\diamond}_{\diamond}(\ccc^{\op} \times \ccc)}(I,I)$$
by adjunction.
\end{proof}

\subsection{Hochschild cohomology and large additive functor categories}\label{hhtop}
Let $k$ be a field. The following definition of Hochschild cohomology of a (possibly large) abelian category $\ddd$ was communicated to the second author by Ragnar Buchweitz, who attributed it to John Greenlees. One considers the (possibly large) abelian category $\Add(\ddd, \ddd)$ of additive functors from $\ddd$ to $\ddd$ and puts
$$HH^n_{\mathrm{top}}(\ddd) = \Ext^n_{\Add(\ddd, \ddd)}(1_{\ddd}, 1_{\ddd}).$$
This subsection is devoted to the proof of the following

\begin{theorem}\label{maingroth}
For a Grothendieck category $\ddd$, we have
$$HH^n_{\mathrm{gro}}(\ddd) \cong HH^n_{\mathrm{top}}(\ddd).$$
\end{theorem}

The theorem is known to hold true for module categories (see \cite{jibladzepirashvili1}, \cite{jibladzepirashvili}), and the proof of the theorem relies heavily on this case, which we first discuss.

For later use, apart from our standard universe $\mathsf{U}$, we introduce another universe $\mathsf{U} \subseteq \mathsf{V}$. As usual, $\mathsf{U}$ is suppressed in the notations.
Let $\AAA$ be a small linear category.
Consider the adjoint pair
$$R: \Add(\Mod(\AAA), \mathsf{V}\text{-} \Mod(\AAA)) \lra \Add(\AAA, \mathsf{V}\text{-} \Mod(\AAA)) \cong \mathsf{V}\text{-} \Mod(\AAA^{\op} \otimes \AAA): F \longmapsto F|_{\AAA}$$
and $$L: \mathsf{V}\text{-} \Mod(\AAA^{\op} \otimes \AAA) \lra \Add(\Mod(\AAA), \mathsf{V}\text{-} \Mod(\AAA)): M \longmapsto M \otimes_{\AAA} -.$$
Let $I \in \mathsf{V}\text{-} \Mod(\AAA^{\op} \otimes \AAA)$ be the identity bimodule and $j: \Mod(\AAA) \lra \mathsf{V}\text{-} \Mod(\AAA)$ the natural inclusion.

\begin{lemma}[see \cite{jibladzepirashvili1}, \cite{jibladzepirashvili}]\label{lemlin}
For $M \in C(\Add(\Mod(\AAA), \mathsf{V}\text{-} \Mod(\AAA)))$, we have
$$\RHom_{\mathsf{V}\text{-} \Mod(\AAA^{\op} \otimes \AAA)}(I, M|_{\AAA}) \cong \RHom_{\Add(\Mod(\AAA), \mathsf{V}\text{-} \Mod(\AAA))}(j, M).$$
\end{lemma}

\begin{proof}
Let $B(I) \lra I$ be the bar resolution of $I$ in $\mathsf{V}\text{-} \Mod(\AAA^{\op} \otimes \AAA)$. Concretely, we have $$B^n(I) = \oplus_{A_0, \dots, A_n} \AAA(A_n,-) \otimes_k \AAA(A_{n-1}, A_n) \otimes_k \dots \otimes_k \AAA(-,A_0).$$ 
Projectivity of $L(B^n(I))$ follows automatically from the adjunction since $R$ is exact. To see that $L(B(I)) \lra L(I) = j$ remains a resolution, it suffices to check its evaluation at an arbitrary $X \in \Mod(\AAA)$. We have
$$L(B^n(I))(X) = \oplus_{A_0, \dots, A_n} X(A_n) \otimes_k \AAA(A_{n-1}, A_n) \otimes_k \dots \otimes_k \AAA(-,A_0),$$ 
so this is precisely the bar resolution of $X$. Finally, we can write 
$$\begin{aligned}
\RHom_{\mathsf{V}\text{-} \Mod(\AAA^{\op} \otimes \AAA)}(I, M|_{\AAA}) & = \Hom_{\mathsf{V}\text{-} \Mod(\AAA^{\op} \otimes \AAA)}(B(I), R(M))\\ & = \Hom_{\Add(\Mod(\AAA), \mathsf{V}\text{-} \Mod(\AAA))}(L(B(I)), M) \\ & = \RHom_{\Add(\Mod(\AAA), \mathsf{V}\text{-} \Mod(\AAA))}(j, M).
\end{aligned}$$
\end{proof}

Obviously, taking $\mathsf{U} = \mathsf{V}$ and $M = j = 1_{\Mod(\AAA)}$ in Lemma \ref{lemlin} yields Theorem \ref{maingroth} for $\ddd = \Mod(\AAA)$.

Now let $\ddd$ be an arbitrary Grothendieck category and choose an equivalence $\ddd \cong \Sh(\UUU) = \Sh(\UUU, \ttt)$ for an additive topology $\ttt$ on a small $\Z$-linear category $\UUU$ (see \S \ref{paraddtop}). From now on, we choose $\mathsf{U} \subseteq \mathsf{V}$ in such a way that $\Mod(\UUU)$ and $\Sh(\UUU)$ are $\mathsf{V}$-small, and we consider the categories $\mathsf{V}\text{-} \Mod(\UUU)$, $\mathsf{V}\text{-} \Sh(\UUU)$. We have a commutative diagram:
$$\xymatrix{{\Mod(\UUU)} \ar[r]^j \ar@/_10pt/[d]_{a} & {\mathsf{V}\text{-} \Mod(\UUU)} \ar@/_10pt/[d]_{a'}\\
{\Sh(\UUU)} \ar[u]_{i} \ar[r]_{j'} & {\mathsf{V}\text{-} \Sh(\UUU).} \ar[u]_{i'}}.$$
The proof consists of three steps, and some remarks on how to get rid of the additional universe $\mathsf{V}$. 

First,  we take an injective resolution $j'a \lra E$ in the $\mathsf{V}$-Grothendieck category $\Add(\Mod(\UUU), \mathsf{V}\text{-} \Sh(\UUU))$. Then the restriction $j'aI \lra EI$ for $I: \UUU \lra \Mod(\UUU)$ yields a functorial choice of injective resolutions $a'(\UUU(-,U)) \lra E(\UUU(-,U))$ in $\mathsf{V}\text{-} \Sh(\UUU)$. By Proposition \ref{newgroth} and Lemma \ref{lemlin}, we have
\begin{equation}\label{step1}
\begin{aligned}
\CC_{\mathrm{gro}}(\mathsf{V}\text{-} \Sh(\UUU))
& \cong \RHom_{\mathsf{V}\text{-} \Mod(\UUU^{\op} \otimes \UUU)}(jI, i'EI) \\
 & \cong \RHom_{\Add(\Mod(\UUU), \mathsf{V}\text{-} \Mod(\UUU))}(j, i'E).
 \end{aligned}
 \end{equation}

For the second step, we note that the localization $(a',i')$ induces a localization
$$a'\circ -: \Add(\Mod(\UUU), \mathsf{V}\text{-} \Mod(\UUU)) \lra \Add(\Mod(\UUU), \mathsf{V}\text{-} \Sh(\UUU))$$ and
$$i'\circ -: \Add(\Mod(\UUU), \mathsf{V}\text{-} \Sh(\UUU)) \lra \Add(\Mod(\UUU), \mathsf{V}\text{-} \Mod(\UUU)).$$
We thus obtain
\begin{equation}\label{step2}
\begin{aligned}
\RHom_{\Add(\Mod(\UUU), \mathsf{V}\text{-} \Mod(\UUU))}(j, i'E) & = \RHom_{\Add(\Mod(\UUU), \mathsf{V}\text{-} \Mod(\UUU))}(j, R(i'-)E) \\ & \cong \RHom_{\Add(\Mod(\UUU), \mathsf{V}\text{-} \Sh(\UUU))}(a'j, E).
\end{aligned}
\end{equation}

For the third step, we use the following localization induced by $(a,i)$:
$$-\circ a: \Add(\Sh(\UUU), \mathsf{V}\text{-} \Sh(\UUU)) \lra \Add(\Mod(\mathsf{U}), \mathsf{V}\text{-} \Sh(\UUU)),$$
$$-\circ i: \Add(\Mod(\mathsf{U}), \mathsf{V}\text{-} \Sh(\UUU)) \lra \Add(\Sh(\UUU), \mathsf{V}\text{-} \Sh(\UUU)).$$
Since both functors are exact, we obtain:
\begin{equation}\label{step3}
\begin{aligned}
\RHom_{\Add(\Mod(\UUU), \mathsf{V}\text{-} \Sh(\UUU))}(a'j, E) & \cong \RHom_{\Add(\Mod(\UUU), \mathsf{V}\text{-} \Sh(\UUU))}(j'a, E)\\ & \cong \RHom_{\Add(\Sh(\UUU), \mathsf{V}\text{-} \Sh(\UUU))}(j', Ei)\\
& \cong \RHom_{\Add(\Sh(\UUU), \mathsf{V}\text{-} \Sh(\UUU))}(j', j'ai)\\
& \cong \RHom_{\Add(\Sh(\UUU), \mathsf{V}\text{-} \Sh(\UUU))}(j', j')
\end{aligned}
\end{equation}

Putting \eqref{step1}, \eqref{step2} and \eqref{step3} together, we now arrive at
$$\CC_{\mathrm{gro}}(\mathsf{V}\text{-} \Sh(\UUU)) \cong \RHom_{\Add(\Sh(\UUU), \mathsf{V}\text{-} \Sh(\UUU))}(j', j').$$
Finally, we need some remarks concerning the universe $\mathsf{V}$. The functor $j': \Sh(\UUU) \lra \mathsf{V}\text{-} \Sh(\UUU)$ a priori does not preserve injective objects, whereas $j: \Mod(\UUU) \lra \mathsf{V}\text{-} \Mod(\UUU)$ does (using the Baer criterium). However, $\Sh(\UUU)$ has enough injectives that are preverved by $j'$. Indeed, for a sheaf $F \in \Sh(\UUU)$, an essential monomorphisms $iF \lra M$ to an injective $M \in \Mod(\UUU)$ actually yields a monomorphism into an injective sheaf, and all involved notions are preserved by $j$. In particular, $j'$ preserves $\Ext$. Thinking of actual extensions, it is then readily seen that 
$$j'\circ -: \Add(\Sh(\UUU), \Sh(\UUU)) \lra \Add(\Sh(\UUU), \mathsf{V}\text{-} \Sh(\UUU))$$
also preserves $\Ext$, whence
$$\Ext^n_{\Add(\Sh(\UUU), \Sh(\UUU))}(1_{\Sh(\UUU)}, 1_{\Sh(\UUU)}) \cong \Ext^n_{\Add(\Sh(\UUU), \mathsf{V}\text{-} \Sh(\UUU))}(j', j').$$
Let us now look at $\CC_{\mathrm{gro}}(\Sh(\UUU))$. If we take for every $U \in \UUU$ a ``special'' injective resolution $a(\UUU(-,U)) \lra E_U$, then the dg category $\UUU_{\mathrm{dg}} \subseteq C(\Sh(\UUU))$ of all these resolutions satisfies
$$\CC_{\mathrm{gro}}(\Sh(\UUU)) \cong \CC_{\mathrm{hoch}}(\UUU_{\mathrm{dg}}).$$
Taking the images of the $E_U$ under $j$ yields a quasi-equivalent dg category, whence
$$\CC_{\mathrm{gro}}(\Sh(\UUU)) \cong \CC_{\mathrm{gro}}(\mathsf{V}\text{-} \Sh(\UUU)).$$
This finishes the proof of Theorem \ref{maingroth}.

\subsection{Mac Lane cohomology of $\Z$-linear categories}

Mac Lane cohomology originated in \cite{maclane} as a cohomology theory for rings $A$ taking values in bimodules. In \cite{jibladzepirashvili}, the authors discovered an incarnation allowing for a natural generalization to Mac Lane cohomology with values in \emph{non-additive} functors $\mathsf{free}(A) \lra \Mod(A)$. We review the situation for a small $\Z$-linear category $\AAA$.

For an abelian group $A$, denote by $Q(A)$ the cube construction of $A$ \cite{maclane}. This is a cochain complex of abelian groups in nonpositive degrees, together with an augmentation $Q^0(A) \lra A$ such that $H^0(Q(A)) \cong A$. For abelian groups $A$ and $B$, there is a natural pairing
$$Q(A) \otimes Q(B) \lra Q(A \otimes B).$$
This allows us to define a differential graded $\Z$-linear category $Q(\AAA)$ with
$Q(\AAA)(A,B) = Q(\AAA(A,B))$ for $A, B \in \AAA$ and composition morphisms
$$Q(\AAA(B,C)) \otimes Q(\AAA(A,B)) \lra Q(\AAA(B,C) \otimes \AAA(A,C)) \lra Q(\AAA(A,C))$$
just like in the ring case.
For $M \in C(\Mod(\AAA^{\op} \otimes \AAA))$, we put 
$$\CC_{\mac'}(\AAA, M) = \CC_{\hoch}(Q(\AAA), M)$$
where the right hand side is Hochschild cohomology of the dg category $Q(\AAA)$ with values in the dg bimodule $M$.

Now consider the inclusion $I: \tilde{\AAA} \lra \Mod(\AAA)$ of a full additive subcategory containing $\AAA$, and a cochain complex $M \in C(\Add(\tilde{\AAA}, \Mod(\AAA)))$. We denote both the restriction of $M$ to $C(\Add(\AAA, \Mod(\AAA)) = C(\Mod(\AAA^{\op} \otimes \AAA))$ and the image of $M$ in $C(\Fun(\tilde{\AAA}, \Mod(\AAA))$ - the category of cochain complexes of non-additive functors - by $M$.

We have:

\begin{theorem}\label{pira} \cite[Theorem A]{jibladzepirashvili}
There is an isomorphism
$$\CC_{\mac'}(\AAA, M) \cong \RHom_{\Fun(\tilde{\AAA}, \Mod(\AAA))}(I, M).$$
\end{theorem}

Consider the following two variants of Mac Lane cohomology:

\begin{definition}\label{defmac}
\begin{enumerate}
\item For a $\Z$-linear category $\AAA$ and $T \in C(\Fun(\AAA, \Mod(\AAA)))$, 
$$\CC_{\mac}(\AAA, T) = \RHom_{\Fun(\AAA, \Mod(\AAA))}(I, T).$$
\item For a $\Z$-linear category $\AAA$ and $T \in C(\Fun(\mathsf{free}(\AAA), \Mod(\AAA)))$, 
$$\CC_{\mac'}(\AAA, T) = \RHom_{\Fun(\mathsf{free}(\AAA), \Mod(\AAA))}(I, T).$$
\end{enumerate}
\end{definition}

\begin{remark}\label{remmac}
The two notions in Definition \ref{defmac} are related in the following way. It $\AAA$ is additive, then $\AAA \cong \free(\AAA)$ and if $T$ and $T'$ correspond under the equivalence $C(\Fun(\AAA, \Mod(\AAA))) \cong C(\Fun(\free(\AAA), \Mod(\AAA)))$, then $$HH^n_{\mac}(\AAA, T) \cong HH^n_{\mac'}(\AAA, T').$$ If $\AAA$ is arbitrary, then $\Mod(\AAA) \cong \Mod(\free(\AAA))$ and if $T$ and $T'$ correspond under the equivalence $C(\Fun(\mathsf{free}(\AAA), \Mod(\AAA))) \cong C(\Fun(\mathsf{free}(\AAA), \Mod(\free(\AAA))))$, then 
$$HH^n_{\mac'}(\AAA, T) \cong HH^n_{\mac}(\free(\AAA), T').$$
By Theorem \ref{pira}, $\CC_{\mac'}(\AAA ,T)$ directly generalizes the earlier definition for $M \in C(\Mod(\AAA^{\op} \otimes \AAA)) \cong C(\Add(\mathsf{free}(\AAA), \Mod(\AAA)))$.
\end{remark}

As proven in \cite{jibladzepirashvili}, Mac Lane cohomology also has a natural interpretation in terms of Hochschild-Mitchel cohomology. Let $\AAA$ be a small (non-linear) category. 
Recall from \cite{baueswirsching} that a natural system $M$ on $\AAA$ is given by abelian groups $M(\lambda)$ associated to the morphisms $\lambda: A \lra B$ of $\AAA$, and morphisms $M(\lambda) \lra M(\lambda')$ associated to compositions $\lambda' = b\lambda a$ in $\AAA$ with $a: A' \lra A$ and $b: B \lra B'$ (satisfying the natural associativity condition). Hochschild-Mitchel cohomology of $\AAA$ with values in $M$ is the cohomology of the natural ``Hochschild type'' complex with
$$\CC^n_{\mathrm{mitch}}(\AAA, M) = \prod_{(\lambda_1, \dots, \lambda_n) \in N_n(\AAA)} M(\lambda_n \dots \lambda_1)$$
where $\xymatrix{{A_0} \ar[r]_{\lambda_1} & {A_1} \ar[r] & {\dots} \ar[r]_{\lambda_n} & {A_n}}$ is a sequence of $\AAA$-morphisms in the nerve of $\AAA$. 
We have
$$\CC_{\mathrm{mitch}}(\AAA, M) \cong \RHom_{\mathsf{Nat}(\AAA)}(\underline{\Z}, M)$$
where $\mathsf{Nat}(\AAA)$ is the abelian category of natural systems on $\AAA$ and $\underline{\Z}$ is the constant natural system with $\underline{\Z}(\lambda: A \lra B) = \Z$. A bifuncor $M: \AAA^{\op} \times \AAA \lra \Ab$ is naturally considered as a natural system with $M(\lambda: A \lra A') = M(A, A')$.

Now we return to the setting of a $\Z$-linear category $\AAA$. For $T \in C(\Fun(\AAA, \Mod(\AAA)))$, consider $T$ as a complex of bifunctors $\AAA^{\op} \times \AAA \lra \Ab$, and hence as a natural system.

\begin{proposition} \cite[Proposition 3.12]{jibladzepirashvili}\label{mitchlink}
We have:
$$\CC_{\mac}(\AAA, T) \cong \CC_{\mathrm{mitch}}(\AAA, T).$$
\end{proposition}

\subsection{Mac Lane cohomology of additive sites}\label{MLadd}
In this subsection, we adapt the notions of the previous subsection to the situation of a linear site.
Let $(\UUU, \ttt)$ be a $\Z$-linear site with additive sheaf category $\Sh(\UUU)$ and canonical functor $u: \UUU \lra \Sh(\UUU)$. We start with an analogue of Proposition \ref{newgroth}. Consider the localization
$$i \circ -: \Fun(\UUU, \Sh(\UUU)) \lra \Fun(\UUU, \Mod(\UUU)),$$
$$a \circ -: \Fun(\UUU, \Mod(\UUU)) \lra \Fun(\UUU, \Sh(\UUU))$$
induced by $i: \Sh(\UUU) \lra \Mod(\UUU)$, $a: \Mod(\UUU) \lra \Sh(\UUU)$.

\begin{proposition}\label{newgrothmac}
We have:
$$\CC_{\mac}(\UUU, R(i \circ -)u) \cong \RHom_{\Fun(\UUU, \Sh(\UUU))}(u,u).$$
Furthermore, for every natural transformation $u \lra F$ in $C(\Fun(\UUU, \Sh(\UUU)))$ for which every $u(U) \lra F(U)$ is an injective resolution, we have:
$$\CC_{\mathrm{mac}}(\UUU, F) \cong \RHom_{\Fun(\UUU, \Sh(\UUU))}(u,u).$$
\end{proposition}

\begin{proof}
The first line immediately follows from the adjunction. Furthermore, by Lemma \ref{lemEF} (with $\AAA = \Z \UUU$) we have $R(i \circ -)(u) \cong iF$ whence the second statement follows.
\end{proof}

We define $\CC_{\mac}(\UUU, \ttt) = \RHom_{\Fun(\UUU, \Sh(\UUU))}(u,u)$.

\subsection{Mac Lane cohomology of exact categories and (bi)sheaf categories}
Let $\ccc$ be an exact $\Z$-linear category with canonical embedding $\iota: \ccc \lra \Lex(\ccc)$. This subsection is parallel to \S \ref{hochbisheaf}. We define Mac Lane cohomology of $\ccc$ to be Mac Lane cohomology of the natural site $(\ccc, \ttt)$ where $\ttt$ is the single deflation topology. Concretely,
\begin{equation}\label{macex}
\CC_{\mac, \ex}(\ccc) = \RHom_{\Fun(\ccc, \Lex(\ccc))}(\iota, \iota).
\end{equation}

We have the following analogue of Proposition \ref{linex}:

\begin{proposition}\label{linexmac}
Let $\AAA$ be a $\Z$-linear category and $\mathsf{free}(\AAA)$ the exact category of finitely generated free $\AAA$-modules with split exact conflations. We have:
$$\CC_{\mac, \mathrm{ex}}(\mathsf{free}(\AAA)) \cong \CC_{\mac'}(\AAA, I).$$
\end{proposition}

\begin{proof}
We have $\Lex(\mathsf{free}(\AAA)) \cong \Mod(\mathsf{free}(\AAA)) \cong \Mod(\AAA)$ (see Remark \ref{remsplit}). Hence the result immediately follows from the definitions.
\end{proof}

Let $I$ denote the identity $\ccc$-bimodule with $I(C',C) = \ccc(C',C)$.
Using the results of \S \ref{sub2var}, we obtain the following expression in terms of sheaves in two variables:

\begin{theorem}\label{MLex}
We have:
$$\CC_{\mac, \ex}(\ccc) \cong  \RHom_{\Fun^{\triangleleft}_{\diamond}(\ccc^{\op} \times \ccc)}(I,I).$$
\end{theorem}

\begin{proof}
Consider $i_1: \Fun^{\triangleleft}_{\triangleleft}(\ccc^{\op} \times \ccc) \lra \Fun^{\triangleleft}(\ccc^{\op} \times \ccc)$, which is isomorphic to $i \circ -: \Fun(\ccc, \Lex(\ccc)) \lra \Fun(\ccc, \Mod(\ccc))$.
The definition \eqref{macex} translates into $\CC_{\mac, \ex}(\ccc) \cong \RHom_{\Fun^{\triangleleft}_{\triangleleft}(\ccc^{\op} \times \ccc)}(I,I)$. From Proposition \ref{key} we further obtain $Ri_1(I) \cong Rj(I)$ for $j: \Fun^{\triangleleft}_{\diamond}(\ccc^{\op} \times \ccc) \lra \Fun^{\triangleleft}(\ccc^{\op} \times \ccc)$, so 
$$\CC_{\mac, \ex}(\ccc) \cong \RHom_{\Fun^{\triangleleft}}(I, Rj(I)) \cong \RHom_{\Fun^{\triangleleft}_{\diamond}(\ccc^{\op} \times \ccc)}(I,I)$$
by adjunction.
\end{proof}

\section{Discussion}\label{sec.4}

To finish the paper, let us now explain informally and without
proofs the motivations behind our various definitions and
constructions.

First of all, our emphasis on abelian and exact categories seems
distinctly old-fashioned; these days, it is much more common to
start with a triangulated category (for example, the derived
category $\ddd(\ccc)$ of an abelian category $\ccc$ instead of the
category $\ccc$ itself). The problem with this approach is that of
course just a triangulated category is not enough -- the category of
exact functors from a triangulated category to itself is not
triangulated. To get the correct endofunctor category, one needs
some enhancement, see e.g. \cite{bondalkapranov}.

When working over a field, a DG enhancement (see \cite{keller1}, \cite{keller8}) would do the job, but at
the cost of technical complications which obscure the essential
content of the theory. Thus a purely abelian treatment is also
useful. Moreover, there is one point where the abelian treatment
should be considerably simpler. Namely, assume given an abelian
category $\ccc$ and another abelian category $\ccc'$ which is a
``square-zero extension'' of $\ccc$ in some sense (for example,
$\ccc$ could be modules over some algebra, and $\ccc'$ could be
modules over a square-zero extension of this algebra). Then we have
a pair of adjoint functors $i_*:\ccc \lra \ccc'$, $i^*:\ccc' \lra
\ccc$, with $i_*$ being exact and fully faithful, and the total
derived functor $L^\hdot i^*$. It turns out that in a rather general
situation, it is the composition $L^1i^* \circ i_*:\ccc \lra \ccc$ of
the first derived functor $L^1i^*$ with the embedding $i_*$ which
serves as tangent space to $\ccc$ inside $\ccc'$. Moreover, taking
the appropriate canonical truncation of the total derived functor,
we obtain a complex of functors with $0$-th homology isomorphic to
the identity $\id$ and the first homology isomorphic to $L^1i^*
\circ i_*$. By Yoneda, this complex represents a class in
$$
\Ext^2(\id,L^1i^* \circ i_*),
$$
and it is this class that should be the Hochschild cohomology class of
the square-zero extension.

Of course, even with the various functor categories introduced in
the present paper, making the above sketch precise requires some
work, and we relegate it to a subsequent paper. Nevertheless, it is
already obvious that the abelian context is essential: if one works
with enhanced triangulated categories, one cannot separate $L^1i^*$
from the total derived functor $L^\hdot i^*$.

When working absolutely, the situation becomes much more complicated
from the technical point of view. DG enhancement is no longer
sufficient; among the theories existing in the literature, the ones
which would apply are either spectral categories, see e.g. \cite{schwedeshipley}, or
$\infty$-categories in the sense of Lurie \cite{lurie}. Both require quite a lot
of preliminary work.

However, surprising as it may be, at least in the simple case
mentioned in the introduction, --- namely that of $\ccc$ being the
category $\Z/p\Z$-vector spaces, --- the correct ``absolute''
endofunctor category of $\ccc$ is very easy to describe.

Namely, let $\Fun(\ccc,\ccc)$ be the category of all functors from
$\ccc$ to itself that commute with filtered direct limits. It
is an abelian category, so that we can take its derived category
$\ddd(\ccc,\ccc)$. Then the triangulated category of ``absolute''
endofunctors of $\ccc$ should be the full triangulated subcategory
$$
\ddd_{add}(\ccc,\ccc) \subset \ddd(\ccc,\ccc)
$$
spanned by functors which are additive. We note that this is
different from the derived category of the abelian category of
additive functors -- indeed, since every additive functor in
$\Fun(\ccc,\ccc)$ is given by tensor product with a fixed vector
space $V \in \ccc$, the latter is just the derived category
$\ddd(\ccc)$. However, there are higher $\Ext$'s between additive
endofunctors in $\ddd(\ccc,\ccc)$ which do not occur in
$\ddd(\ccc)$. For example, for any vector space $V$, consider the
tensor power $V^{\otimes p}$, and let $\sigma:V^{\otimes p} \lra
V^{\otimes p}$ be the longest cycle permutation. Then one can
consider the complex
$$
\xy
\xymatrix{ {(V^{\otimes p})_\sigma} \ar[rr]^-{\id + \sigma + \dots +
 \sigma^{p-1}} && {(V^{\otimes p})^\sigma}}
 \endxy
$$
and it is easy to show that the homology of this complex is
naturally isomorphic to $V$ both in degree $1$ and in degree
$0$. The complex is functorial in $V$, thus defines by Yoneda an
element in
$$
\Ext^2(\id,\id)
$$
in the category $\ddd_{add}(\ccc,\ccc)$. This element is in fact
non-trivial, and corresponds to the square-zero extension $\Z/p^2\Z$
of the field $\Z/p\Z$.

The category $\ddd_{add}(\ccc,\ccc)$ is the simplest example of a triangulated
category of ``non-additive bimodules'' whose importance for Mac Lane
homology and topological Hochschild homology has been known since
the pioneering work of Jibladze and Pirashvili in the 1980ies, 
 see e.g. \cite{jibladzepirashvili1}, \cite{jibladzepirashvili}, \cite{loday13}. What we would like to do is to obtain a similar
category for an abelian category $\ccc$ which is not the category of
modules over an algebra (and for example does not have enough
projectives). Our best approximation to the correct category is
$\Fun^{\triangleleft}_{\diamond}(\ccc^{\op} \times \ccc)$. We
believe that it does give the correct absolute Hochschild
cohomology. However, one significant problem with this category is
that it does not have a natural tensor structure -- this is not
surprising, since its very definition is asymmetric. When $\ccc$ is
the category of $\Z/p\Z$-vector spaces, the triangulated category
$\ddd_{add}(\ccc,\ccc)$ does have a tensor structure given by the
composition of functors; however, our
$\Fun^{\triangleleft}_{\diamond}(\ccc^{\op} \times \ccc)$ gives
something like a DG enhancement for $\ddd_{add}(\ccc,\ccc)$, and the
tensor product appears to be incompatible with this DG
enhancement. Perhaps this is unavoidable, and one should expect
$\ddd_{add}(\ccc,\ccc)$ to be a genuinely ``topological''
triangulated category, with a spectral enhancement instead of a DG
one. Be it as it may, in practice, it is the tensor structure that
produces the Gerstenhaber bracket and other higher structures on
Hochschild cohomology, and it is thus unclear whether our absolute
Hochschild cohomology possesses these structures. The deformation
theory on a purely abelian level seems to work, though; we plan to
return to this in the future.

\def\cprime{$'$} \def\cprime{$'$}
\providecommand{\bysame}{\leavevmode\hbox to3em{\hrulefill}\thinspace}
\providecommand{\MR}{\relax\ifhmode\unskip\space\fi MR }
\providecommand{\MRhref}[2]{%
  \href{http://www.ams.org/mathscinet-getitem?mr=#1}{#2}
}
\providecommand{\href}[2]{#2}

\end{document}